\documentclass{amsart}

\usepackage{amssymb, amsmath}
\usepackage[a4paper,top=3cm,bottom=3cm,left=1.5 cm,right=1.5cm]{geometry}

\usepackage[T1]{fontenc}
\usepackage[utf8]{inputenc}
\usepackage[english]{babel}
\usepackage{lipsum}
\usepackage[dvips]{graphicx}
\usepackage{alltt,graphicx,graphpap,color}
\usepackage{latexsym}
\usepackage{amsxtra}
\usepackage{amscd}
\usepackage{amsthm}
\usepackage{mathrsfs}
\usepackage{mathtools}
\usepackage{quoting}
\usepackage{indentfirst}
\usepackage{verbatim}
\quotingsetup{font=small}
\usepackage{ytableau}
\usepackage{xcolor}

\DeclareMathOperator{\I}{Id}

\newtheorem{Theorem}{Theorem}[section]
\newtheorem{Definition}{Definition}[section]

\newtheorem{Remark}{Remark}[section]
\newtheorem{Proposition}{Proposition}[section]
\newtheorem{Lemma}{Lemma}[section]
\newtheorem{Corollary}{Corollary}[section]

\title{ Varieties of group-graded algebras
	of proper central exponent greater than two }

\author{Francesca S. Benanti}
\address{Dipartimento di Matematica e Informatica, Universit\`a di Palermo, Via Archirafi 34, 90123, Palermo, Italy}
\email{francescasaviella.benanti@unipa.it}

\author{Angela Valenti}
\address{Dipartimento di Ingegneria, Universit\`a di Palermo, Viale delle Scienze, 90128, Palermo, Italy}
\email{angela.valenti@unipa.it}

\thanks{The authors are partially supported by INDAM-GNSAGA of Italy
	and  by  Fondo Finalizzato alla Ricerca dell'Università degli Studi di Palermo  (FFR) – Anno 2025.}



\begin{document}

	\begin{abstract}
		Let $F$ be a field of characteristic zero and let $ \mathcal V $ be a variety of associative $F$-algebras graded by a finite abelian group $G$.  To  a variety $ \mathcal V $  is associated a numerical sequence called the sequence  of proper central $G$-codimensions, $c^{G,\delta}_n(\mathcal V), \, n \ge 1.$ Here $c^{G,\delta}_n(\mathcal V)$ is the dimension of the space of multilinear proper central $G$-polynomials in $n$ fixed variables of  any algebra $A$ generating the variety	$\mathcal V.$
      Such sequence gives information on the growth of the proper central $G$-polynomials of $A$ and in \cite{LMR} it was  proved that $exp^{G,\delta}(\mathcal V)=\lim_{n\to\infty}\sqrt[n]{c_n^{G,\delta}(\mathcal V)}$ exists and is an integer called the proper central $G$-exponent.

		The aim of this paper is to characterize the varieties of associative
		$G$-graded algebras  of proper central $G$-exponent greater than two.
        To this end we construct a finite list of $G$-graded algebras and we prove that $exp^{G,\delta}(\mathcal V) >2$ if and only if at least one of the algebras belongs to $\mathcal V$.

        Matching this result with the characterization of the varieties of almost polynomial growth given in \cite{GLP},  we obtain a characterization of the varieties of proper central $G$-exponent equal to two.
	\end{abstract}
	
	\maketitle
	
	\section{Introduction}
	
	Let $G $ be a finite abelian group and let $A$ be an associative $G$-graded algebra over a field $F$ of characteristic zero. If we denote by $F\langle X,G\rangle$ the free associative $G$-graded algebra of countable rank over $F,$  then a $G$-polynomial $f \in F\langle X,G\rangle $ is a central $G$-polynomial of $A$ if $f$ takes values in $Z(A),$ the center of $A.$ In case $f$ takes a non zero values in $A,$ we say that $f$ is a proper central $G$-polynomial otherwise we say that $f$ is a $G$-polynomial identity of $A.$
	Notice that if $G$ is trivial we obtain the well-known notions of central polynomial, proper central polynomial and polynomial identity of the associative algebra $A.$
	
	It is known that $M_n(F),$ the algebra of $ n \times n$ matrices over $F,$ has proper central polynomials, as proved independently by  Formanek and Razmyslov in \cite{Formanek} and \cite{Razmyslov} confirming a famous conjecture of Kaplansky.  This result is highly non-trivial since even if an algebra has a non-zero center, the existence of proper central
	polynomials is not granted. For instance, it is well-known (see \cite{GZ2018}) that
	the algebra of upper block triangular matrices has no proper central polynomials.
	
	In this paper in order to get information about how many proper central $G$-polynomials a given $G$-graded algebra has,  we consider, for any $n \ge 1 ,$ the space $P_n^G$ of multilinear $G$-polynomials  in $n$ fixed variables. \  We attach to it three numerical sequences: $c^G_n(A)$, the dimension of  $P_n^G$ modulo the $G$-polynomial identities of A,
	$c^{G,z}_n(A)$,  the dimension of  $P_n^G$ modulo the central $G$-polynomials of $A$  and
		$$
	c_n^{G,\delta}(A)= c^G_n(A)- c_n^{G,z}(A),
	$$
	the dimension of the space of multilinear proper central $G$-polynomials in $n$ fixed variables.
	They are called the $G$-graded, central $G$-graded and proper central $G$-graded codimension sequence of $A$, respectively.
	
		The asymptotic behavior of the $G$-graded codimension sequence was extensively studied in the past years. In fact, in \cite{GR} it was proved that if $A$ satisfies a non-trivial ordinary polynomial identity (i.e., $A$ is a PI-algebra), then $c^G_n(A)$ is exponentially bounded, moreover in \cite{AG}, \cite{AGL} and \cite{GL} the authors proved the existence and the integrability of the $G$-exponent, $exp^G(A)=\lim_{n\to\infty}\sqrt[n]{c_n^{G}(A)}$, also for non-abelian groups. This result was generalized in the setting of $H$-module algebras, where $ H$ is a semisimple finite dimensional Hopf algebra (see \cite{Ka}).
	
	In \cite{LMR} it was proved that for a PI-algebra $A$, also $exp^{G,z}(A)=\lim_{n\to\infty}\sqrt[n]{c_n^{G,z}(A)}$ and $exp^{G,\delta}(A)=\lim_{n\to\infty}\sqrt[n]{c_n^{G,\delta}(A)}$ exist and are integers. They are called the central $G$-exponent and the proper central $G$-exponent of $A$,  respectively.
	
	In case  $G$ is the trivial group we obtain the ordinary codimension,
	the central codimension,
	and the proper central codimension.
	The existence and the integrability of the exponent,  the central exponent and  the proper central exponent was proved in \cite{GZ1}, \cite{GZ2}, \cite{GZ2018}, \cite{GZ2019}.
	
	In the language of the varieties, if $\mathcal V=var^G(A)$ is the variety of $G$-graded algebras generated by a $G$-graded algebra $A$, then we write $exp^G(\mathcal V)=exp^G(A)$,  $exp^{G,z}(\mathcal V)=exp^{G,z}(A)$ and $exp^{G,\delta}(\mathcal V)=exp^{G,\delta}(A)$.
	
	Recently in \cite{GLP} the authors classified the algebras for which the sequence of proper central codimensions has almost polynomial growth. Moreover they obtained a classification also in the case of group graded algebras. A characterization of the varieties of  associative algebras of proper central exponent greater than two was obtained in \cite{BV}. As a consequence the authors achieved a description for the varieties with proper central exponent equal to two.

	The purpose of this paper is to characterize the varieties of $G$-graded algebras having proper central $G$-exponent
	greater than two. To this end, we shall explicitly exhibit a  list of
	algebras in order to prove that a variety $\mathcal{V}$ has proper central $G$-exponent
	greater than two if and only if at least one of these  belongs to $ \mathcal{V}.$  As a consequence of this result and  of Theorem 3.1 in \cite{GLP},
	we shall obtain a characterization of the varieties of proper central $G$-exponent equal to two.

	\section{Preliminaries and basic results}
	Let $F$ be a field of characteristic zero and  $G$ a finite abelian group.
	Recall that an algebra $A$ is a $G$-graded algebra if $A$ can be written as a direct sum of vector spaces $A=\bigoplus_{g\in G}A^{(g)}$
	such that $A^{(g)}A^{(h)}\subseteq A^{(gh)}$, for all $g,h \in G$. The subspaces $A^{(g)}$ are the  homogeneous components of $A$ of degree $g$ and  an element $a\in A^{(g)}$ is called homogeneous of degree $g=|a|_A$.
	Let $F\langle X \rangle$ be the free associative algebra on a countable set $X=\{x_1, x_2, \ldots \}$. One can define on such an algebra a $G$-grading in a natural way: write $X = \cup_{g\in G}X^{(g)}$, where $X^{(g)}= \{x_1^g, x_2^g, \ldots\}$ are disjoint sets of elements of homogeneous degree $g$. Let $\mathcal{F}^{(g)}$ be the subspace of $F\langle X\rangle$ spanned by all monomials in the variables of $X$ having homogeneous degree $g$.
	Then $F\langle X,G\rangle= \oplus_{g \in G}\mathcal{F}^{(g)}$ is the free associative $G$-graded algebra of countable rank over $F$.
	
	Write $G = \{g_1,\ldots, g_s\}$.
	A $G$-graded polynomial, or simply a $G$-polynomial, $f =f(x_1^{g_1} ,\ldots,x_{t_1}^{g_1},\ldots,x_1^{g_s} ,\ldots,x_{t_s}^{g_s})$ of $F\langle X,G\rangle$ is a $G$-polynomial identity, or simply a graded identity, of $A$, and we write $f \equiv 0$, if
	$$f(a_1^{g_1} ,\ldots,a_{t_1}^{g_1},\ldots,a_1^{g_s} ,\ldots ,a_{t_s}^{g_s})=0, $$
	for all $a_i^{g_i} \in A^{(g_i)}$ , $t_i \geq 0$, for all $1 \leq i \leq s$.

	Let $Id^G(A) = \{f \in F\langle X,G\rangle \, | \, f \equiv 0 \, \,\mathrm{on} \, \, A\}$ be the ideal of graded identities of $A$. It is easily seen that $Id^G(A)$ is a $T_G$-ideal, i.e., it is an ideal invariant under all graded endomorphisms of $F\langle X,G\rangle$. Notice that if, for some $i \geq 1,$ we set $x_i = x_i^{g_1} + \cdots + x_i^{g_s},$  then $ F\langle X\rangle$ is naturally embedded into $F\langle X,G\rangle$ so that we can look at the ordinary
	identities of A as a special kind of graded identities. Since char $F =0$ then  $Id^G(A)$ is determined by the multilinear $G$-polynomials it contains. Thus, for all $n \geq 1$, one can define
	$$
	P^G_n =span_F\{x_{\sigma(1)}^{g_{i_1}}\cdots x_{\sigma(n)}^{g_{i_n}} \, | \, \sigma \in S_n, \, g_{i_1} ,\ldots,g_{i_n} \in G\}
	$$
	as the space of multilinear $G$-polynomials in the graded variables $x_{1}^{g_{i_1}},\ldots ,x_{n}^{g_{i_n}}$ with $g_{i_j}\in G$. Here $S_n$ stands for the symmetric group of order $n.$
	
	An element $f(x_1^{g_1} ,\ldots,x_{t_1}^{g_1},\ldots,x_1^{g_s} ,\ldots,x_{t_s}^{g_s})$ of $F\langle X,G\rangle$ is a central $G$-polynomial of $A$ if, for all homogeneous elements $a_{_i}^{g_j} \in A^{(g_j)},$ $f(a_1^{g_1} ,\ldots,a_{t_1}^{g_1},\ldots,a_1^{g_s} ,\ldots ,a_{t_s}^{g_s}) \in Z(A).$  If $f$ takes only the zero value, then it is clear that $f$ is a $G$-polynomial identity of $A$, otherwise we say that $f$ is a proper central $G$-polynomial of $A.$
    Let $Id^{G,z}(A) = \{f \in F\langle X,G\rangle  \, | \,  f \,\,\, \mathrm{is\, a \, central} \,\,  G-\mathrm{polynomial \, of \, } A \}.$ Then $Id^{G,z}(A)$ is a $T_G$-space of $F\langle X,G\rangle ,$ i.e., a vector space invariant under all $G$-graded endomorphisms of the free $G$-graded algebra.

	It is well known that in characteristic zero every $T_G$-ideal is completely determined by its multilinear polynomials. Thus it is reasonable to consider the
	quotient spaces, for $n=1, 2, \ldots$
	$$
	P_n^G(A)=\frac{P^G_n }{P^G_n \cap Id^G(A)}, \quad P^{G,z}_n (A)= \frac{P^G_n }{P^G_n \cap Id^{G,z}(A)}, \quad P^{G,\delta}_n (A)=\frac{P^G_n \cap Id^{G,z}(A)}{P^G_n \cap Id^{G}(A)}.
	$$
	Notice that $P^{G,\delta}_n (A)$ corresponds to the space of multilinear proper central $G$-polynomials of $A$ in $n$ fixed homogeneous variables. For $n \ge 1$ let's denote by
	$$c^G_n(A)= \mathrm{dim}_F P_n^G(A), \quad c^{G,z}_n(A)=\mathrm{dim}_F P_n^{G,z}(A),  \quad c^{G,\delta}_n(A)= \mathrm{dim}_F P_n^{G,\delta}(A)$$
	the sequences of $G$-codimensions, central $G$-codimensions and proper central $G$-codimensions of $A$, respectevely.
	
	Since the above  codimension sequences do not change by extension of the base field, from now on we assume that $F$ is algebraically closed.
	
	In \cite{GL}, \cite{AGL}, \cite{AG} and \cite{LMR} it was proved that if $A$ is a $G$-graded PI-algebra
	the following three limits
	$$
	exp^G(A)=\lim_{n\to\infty}\sqrt[n]{c_n^G(A)},\quad  exp^{G,z}(A)=\lim_{n\to\infty}\sqrt[n]{c_n^{G,z}(A)}, \quad exp^{G,\delta}(A)=\lim_{n\to\infty}\sqrt[n]{c_n^{G,\delta}(A)}
	$$
	exist and are non negative integers called the $G$-exponent, the central $G$-exponent and the proper central $G$-exponent  of the algebra $A$, respectively. We also call $exp^{G,\delta}(A)$ simply the  $(G, \delta)$-exponent.
	They can be calculated in an explicit way, and here we need to recall how to compute  the $G$-exponent and the proper central $G$-exponent.
	
	First we have to establish a well-known general setting that one can use in order to reduce the
	problem of computing the $T_G$-ideal of graded identities of any $G$-graded algebra to that
	of the so-called Grassmann envelope of a suitable finite dimensional $G \times \mathbb{Z}_2$-graded
	algebra, where $\mathbb{Z}_2$ is the cyclic group of order two in additive notation.
	Let $E$ denote the infinite dimensional Grassmann algebra generated by the elements
	$1, e_1, e_2, \ldots$ subject to the relations $e_ie_j = -e_j e_i$, for all $i,
	j.$ $E$ has a natural structure of superalgebra $  E = E^{(0)} \oplus E^{(1)},$  where
	$E^{(0)} = \mathrm{span}_F\{e_{i_1} \cdots e_{i_{2k}} \, | \, 1 \le {i_1} < \cdots < {i_{2k}}, k \ge 0\}$
	and
	$E^{(1)} = \mathrm{span}_F \{e_{i_1} \cdots e_{i_{2k+1}} \, | \, 1 \le {i_1} < \cdots < {i_{2k+1}}, k \ge 0\}.$
	Moreover, let $A = \bigoplus_ {(g,i) \in G \times \mathbb{Z}_2 }A^{(g,i)}$
	be a $G \times \mathbb{Z}_2$-graded algebra. Then $A$ has an
	induced $\mathbb{Z}_2$-grading, $A = A^{(0)} \oplus A^{(1)},$ where $A^{(0)} = \oplus_{g \in G }A^{(g,0)}$
	and $A^{(1)} = \oplus_{g \in G }A^{(g,1)}$ and
	an induced $G$-grading $A = \bigoplus_{g \in G }A^{(g)}$
	where $A^{(g)} = A^{(g,0)} \oplus A^{(g,1)}$, for all $g \in G$.
	The Grassmann envelope of $A$ is defined as
	$E(A)=(E^{(0)} \otimes A^{(0)}) \oplus (E^{(1)} \otimes A^{(1)}),$
	it has a natural
	$G$-grading induced by the one of $A$ by setting $E(A)^{(g)} = (E^{(0)}\otimes A^{(g,0)}) \oplus (E^{(1)}\otimes A^{(g,1)}),$ for all $g \in G$.
	
	A famous theorem of Kemer
	assert that an arbitrary algebra satisfying a non-trivial polynomial identity
	over a field of characteristic zero has the same identities as the Grassmann
	envelope $E(B)$ of a finite dimensional $\mathbb{Z}_2$-graded algebra $B$ (see \cite{kemer}). This
	result was extended in the setting of graded algebras  in the following theorem proved separately by Aljadeff
	and Belov in \cite{AKB} and Sviridova in \cite{Sv}.

	\begin{Theorem}  Let $A$ be a $G$-graded algebra satisfying a non-trivial ordinary polynomial
		identity. Then there exists a finite dimensional $G \times \mathbb{Z}_2$-graded algebra $B$ such that
		$Id^G(A) = Id^G(E(B)).$
	\end{Theorem}
	It is worth mentioning that in \cite{AKB} the result was proved also for non-abelian groups.
	
	Notice that if $A$ is a finite dimensional $G$-graded algebra, then we can regard $A$ as a $G\times \mathbb{Z}_2$-graded algebra with trivial $\mathbb{Z}_2$-grading
	and we get that   $\I^G(A)= \I^G(E^{(0)}\otimes A)= \I^G(E(A))$. Thus we may consider $A$ instead of $E(A)$.
	
	The computation of  the $G$-exponent and  the   $(G, \delta)$-exponent is as follows. According to the previous theorem  we may assume that $A =E(B)$ with $B$ a finite dimensional  $G \times \mathbb{Z}_2$-graded algebra.
	By the Wedderburn-Malcev decomposition (see \cite{CR}), we write $B = \bar B + J,$ where  $\bar B$ is a maximal semisimple $G \times \mathbb{Z}_2$-graded
	subalgebra of $B$ and  $J = J(B)$ is the Jacobson radical of $B,$
	that is a graded ideal (see \cite{CM} and \cite{SV}). Also we can write $\bar B= B_1\oplus \cdots \oplus B_q$, as a decomposition into simple $G \times \mathbb{Z}_2$-graded algebras.
	
	Recall that a classification of finite dimensional simple $G\times \mathbb{Z}_2$-graded algebras  is given in the following theorem proved by Bahturin, Sehgal and Zaicev in \cite{BSZ}.
	
	\begin{Theorem} \label{BSZ}
		Let $B$ be a finite dimensional simple $G \times \mathbb{Z}_2$-graded algebra. Then there exist a subgroup $H$ of $G \times \mathbb{Z}_2$, a 2-cocycle $\alpha: H\times H \to F^*,$ where the action of $H$ on $F$ is trivial, an integer $m$ and an $m$-tuple $(g_1=1_G, g_2, ..., g_m) \in (G \times \mathbb{Z}_2)^m$ such that $B$ is isomorphic as $G \times \mathbb{Z}_2$-graded algebra to
		$$R = F^{\alpha}H \otimes M_m(F),$$
		where $R^{(g)}= \mathrm{span}_F\{b_h\otimes e_{ij} \, | \, g=g^{-1}_ihg_j\}.$ Here $b_h \in F^{\alpha}H$ is a representative of $h \in H.$
	\end{Theorem}
	
	In order to estimate the
	$G$-exponent and the  proper central $G$-exponent of $E(B),$ we recall that a semisimple
	$G \times \mathbb{Z}_2$-graded subalgebra
	$C=B_{i_1}\oplus \cdots \oplus B_{i_k}$, where the $B_{i_h}$'s are all distinct,  is an admissible subalgebra of $B$
	if $B_{i_1}J \cdots J B_{i_k}\ne 0$. Moreover we say that $C$ is a centrally admissible subalgebra of $B$,
	if there exists a proper central $G$-polynomial $f$ of $E(B)$ having a non-zero evaluation involving  at least one element of each
	$E(B_{i_h})$, $1\le h \le k$.
	In \cite{GL} it was proved that $exp^G(E(B))$ is the maximal dimension of an admissible subalgebra of $B$ and in \cite{LMR}  it was showed that $exp^{G,\delta}(E(B))$ is the maximal dimension of a
	centrally admissible subalgebra of $B$.

	Recall that if $\mathcal V=var^G(A)$ is the variety of $G$-graded algebras generated by a $G$-graded algebra $A$, then we write $exp^G(\mathcal V)=exp^G(A)$
	and $exp^{G,\delta}(\mathcal V)=exp^{G,\delta}(A)$. Hence the $G$-growth and the $(G,\delta)$-growth of $\mathcal V$ are the growth of the $G$-codimensions  and the growth  of the proper central  $G$-codimensions of the algebra $A$, respectively.  It is worth to mention that the  numerical sequences $c^G_n(A)$ and $c^{G,\delta}_n(A)$  either grow exponentially or are polynomially bounded.
	Hence no intermediate growth is allowed.
	Moreover
	a variety $\mathcal V$ of $G$-graded algebras has almost polynomial $G$-growth if $\mathcal V$ grows exponentially but any
	proper subvariety grows polynomially. Analogously a variety $\mathcal V$ of $G$-graded algebras has almost polynomial $(G,\delta)$-growth if $exp^{G,\delta}(\mathcal V) \ge 2$  and for any proper subvarieties $\mathcal W \subset \mathcal V$ we have that $exp^{G,\delta}(\mathcal V) \le 1.$
	
	In \cite{V} the varieties of $G$-graded algebras of almost polynomial $G$-growth were classified and in \cite{IM} the authors characterized the varieties of $G$-graded
	algebras of $G$-exponent greater than 2. Recently in \cite{GLP}  the varieties of $G$-graded algebras having almost polynomial $(G,\delta)$-growth were classified.

	\section{Graded algebras of small $(G,\delta)$-exponent}
	
	In this section we will introduce some suitable $G$-graded algebras that
	will allow us to prove the main result of this paper.
	We start by establishing some notations and definitions.
	
	Let $M_n(A)$ be the algebra of $n \times n$ matrices over an $F$-algebra $A$ and let $\underline{g} =(g_1,\ldots,g_n)\in G^n$ be an arbitrary $n$-tuple of elements of $G$. Then $\underline{g} =(g_1,\ldots,g_n)$ defines an elementary $G$-grading on $M_n(A)$
	if  the matrix units
	$e_{ij}$ are homogeneous elements of degree $g_i^{-1}g_j.$
	We shall denote by $M_n(A)^{(g_1,\ldots,g_n)}$ or simply by $M_n(A)^{\underline{g}}$ the  algebra $M_n(A)$ with elementary $G$-grading induced by $\underline{g}$.
	Now, if $\underline{g}'$ is obtained from $\underline{g}$ by multiplying  all the elements by a fixed element of $G$, we obtain $M_n(A)^{\underline{g}}= M_n(A)^{\underline{g}'}$.
	Hence, throughout the paper, we shall consider gradings on $M_n(A)^{\underline{g}}$
	with $g_1=1_G$.
	Moreover, if $B$ is a graded subalgebra of $M_n(A)^{\underline{g}}$ the induced grading on $B$ will also be called elementary.
	
	Now let us consider $M_2(F)$, the algebra of $2 \times 2$ matrices endowed with an
	elementary $G\times \mathbb{Z}_2$-grading. It is clear that the possible $G\times \mathbb{Z}_2$-gradings are given by the pairs $((1_G,0),(g,i))$, with $g\in G$, $i\in \{0,1\}$,  then any elementary $G\times \mathbb{Z}_2$-grading is uniquely determined by the
	homogeneous degree of $e_{12}$ and thus we denote by $M_2(F)^{g,i}$ the algebra of $2 \times 2$
	matrices with elementary $G\times \mathbb{Z}_2$-grading induced by $(g, i).$ Finally we define  $\mathcal{A}^{g,i}_1$
	as the Grassmann envelope of $M_2(F)^{g,i}$.
	We observe that  $\mathcal{A}^{g,0}_1=E(M_2(F)^{g,0})$ satisfies the same $G$-polynomial identities of $M_2(F)^{(1,g)}$ and $\mathcal{A}^{g,1}_1=E(M_2(F)^{g,1})$ satisfies the same $G$-polynomial identities of  $M_{1,1}(E)^{(1,g)}=\begin{pmatrix}
		E^{(0)} & E^{(1)}   \\
		E^{(1)}  & E^{(0)}
	\end{pmatrix}^{(1,g)}$.
	
	Let $C_n=\langle g \rangle$ denote the cyclic subgroup generated by an element $g\in G$ of order $n$.
	The group algebra $FC_n$ has a natural $C_n$-grading $FC_n = \oplus_{i=0}^{n-1} FC_n^{(g^i)}$, where $FC_n^{(g^i)}= F{g^i}$,
	$0 \leq i \leq n-1$. It is clear that $FC_n$  can be regarded as a $G$-graded algebra by
	setting $FC_n^{(g')} = 0$, for all $g'\notin \langle g \rangle$. For any prime $p$ greater than two, we denote by $\mathcal{A}_2^p$ the graded algebra $F{C}_p$.
	Moreover let $\mathcal{A}_3=F{C}_4$.
	
	If $g \in G$ is an element of order 4, let $C_{4,1}$ be the cyclic subgroup of $C_4 \times \mathbb{Z}_2$ generated by $(g, 1)$ and $\mathcal{A}_4=E(FC_{4,1})$ the Grassmann envelope of
	$FC_{4,1}$ endowed with its natural $C_{4}$-grading.
	
	Furthermore, if there exist $g, h \in G$ distinct elements of order 2, we let
	$H_{i,j}=<(g,i),(h,j)>$
	be the subgroup of $G \times \mathbb{Z}_2$ generated by $(g, i)$ and $(h, j),$ $i,j \in \mathbb{Z}_2.$ We denote by $\mathcal{A}_5^{i,j}=E(F^\alpha H_{i,j})$  the Grassmann envelope of  $F^\alpha H_{i,j}$, for some cocycle $\alpha$.

	If $\underline{g}=(1_G, g_1,\ldots ,g_{n-1})\in G^n$ and $A$ is an $F$-algebra  we denote by $UT_n^{\textbf{g}}(A)$ the algebra of $n \times n$ upper-triangular matrices endowed with the elementary $G$-grading induced by $\textbf{g}=(1_G, g_1,g_1g_2, g_1g_2g_3, \ldots ,g_1\cdots g_{n-1})$.

	We denote
	by $ \mathcal{A}_6^{g_1,g_2,g_3}$ the subalgebra of $UT_4(F)^{\textbf{g}}$ with $a_{11}=a_{44}$  and $\underline{g}=(1_G, g_1,g_2 ,g_3) \in G^4$, and by $\mathcal{A}_7^{g_1,g_2,g_3, g_4}$  the subalgebra of $UT_5(F)^{\textbf{g}}$ with $a_{11}=a_{55}=0$ and $\underline{g}=(1_G, g_1,g_2 ,g_3,g_4) \in G^5.$
	Moreover, $\mathcal{A}_{8}^{g_1,g_2} $ will be  the graded subalgebra of $UT_3^{\textbf{g}}(E)$ of elements $\begin{pmatrix}
		a & x & y  \\
		0& b & z  \\
		0 & 0 & a
	\end{pmatrix}$ with $ b\in E^{(0)}$ and $\underline{g}=(1_G, g_2, g_1) \in G^3$.
	\noindent
	\smallskip

	In what follows we shall
	denote by $1=1_G$ the unit element of $G,$   by $e$ the unit element of $G\times \mathbb{Z}_2$ and by $g^0$ and $g^1$ the elements  $(g,0)$, $(g,1) \in G\times \mathbb{Z}_2,$ respectively.

	For an ordered set of positive integers $(d_1,\ldots,d_m)$, let
	$
	UT(d_1,\ldots ,d_m)$
	denote the algebra of  upper block-triangular
	matrices over $F$ of size $d_1,\ldots,d_m$ (see \cite{GZbook}).
	Clearly the algebra of upper block triangular matrices  admits an elementary grading. In fact, the embedding of such an algebra into a full matrix algebra with an elementary grading makes it a homogeneous subalgebra.
	We consider the following algebras

	$$
	M= \left \{\begin{pmatrix}
		0 & a & b & d & e  \\
		0 & f & g & h & l \\
		0 & 0 & u & v & m \\
		0 & 0 & v & u & n \\
		0 & 0 & 0 & 0& 0
	\end{pmatrix} \,| \, \, a, \ldots,n \in F  \right \} \subseteq UT(2,3),
	$$
	$$
	N= \left \{\begin{pmatrix}
		0 & a & b & d & e  \\
		0 & u & v & f & g \\
		0 & v & u & h & l \\
		0 & 0 & 0 & m & n \\
		0 & 0 & 0 & 0& 0
	\end{pmatrix} \,| \,\,  a, \ldots,n \in F\right \}\subseteq UT(3,2).
	$$
	
	\noindent
We denote by $\mathcal{A}_{9}^{g_1,g_2,g_3}=E(M^{(e,g_1^0 ,(g_1g_2)^0, (g_1g_2)^{1},(g_1g_2g_3)^0)})$ and $\mathcal{A}_{10}^{g_1,g_2,g_3}= E(N^{(e,g_1^0 ,g_1^1, (g_1g_2)^0,(g_1g_2g_3)^0)})$ the Grassmann envelopes of    $M$ and $N$ with elementary $G\times \mathbb{Z}_2$-grading induced by $(e,g_1^0 ,(g_1g_2)^0, (g_1g_2)^{1},(g_1g_2g_3)^0)$ and by $(e,g_1^0 ,g_1^1, (g_1g_2)^0,(g_1g_2g_3)^0)$, respectively.

	Moreover, let
	$$
	P= \left \{\begin{pmatrix}
		a & b & d & e  \\
		0 & u & v  & f \\
		0 & v & u & g \\
		0 & 0 &  0  & a
	\end{pmatrix} \, | \,\, a, \ldots,g \in F \right \}\subseteq UT(1,2,1).
	$$
	
We denote by	$\mathcal{A}_{11}^{g_1,g_2}$ and $\mathcal{A}_{12}^{g_1,g_2}$  the Grassmann envelope of $P$ with elementary $G\times \mathbb{Z}_2$-grading induced by $(e,g_1^0 ,g_1^1, (g_1g_2)^0)$ and $(e,g_1^0 ,g_1^1, (g_1g_2)^1)$, respectively.

	\medskip
	We now collect the previous $G$-graded algebras in the following list:
	
	\bigskip
	\begin{enumerate}
		\item[ ]  $\mathcal{A}^{g,i}_1=E(M_2(F)^{g,i})$,   $g \in G$, $i\in \mathbb{Z}_2$;
		\medskip
		\item[ ] $\mathcal{A}_2^p=F{C}_p$, $p$ prime such that $p\, | \, \mid G \mid$;
		\medskip
		\item[ ] $\mathcal{A}_3=F{C}_4$;
		\medskip
		\item[ ] $\mathcal{A}_4=E(F{C}_{4,1})$;
		\medskip
		\item[ ] $\mathcal{A}_5^{i,j}=E(F^\alpha H_{i,j})$,  for some cocycle $\alpha$ and  $i,j \in \mathbb{Z}_2$;
		\medskip
		\item[ ]
		$ \mathcal{A}_6^{g_1,g_2,g_3}=\{ (a_{ij}) \in UT_4(F)^{\textbf{g}} \, | \, a_{11}=a_{44}\},$  where $\underline{g}=(1_G, g_1,g_2 ,g_3) \in G^4$;
		
		\medskip
		\item[ ]
		$\mathcal{A}_7^{g_1,g_2,g_3, g_4}= \{ (a_{ij}) \in UT_5(F)^{\textbf{g}} \, | \, a_{11}=a_{55}=0,\}$ where  $\underline{g}=(1_G, g_1,g_2 ,g_3,g_4) \in G^5$;
		
		\medskip
		\item[ ]
		$\mathcal{A}_{8}^{g_1,g_2} =  \{ (a_{ij}) \in UT_3(E)^{\textbf{g}} \, | \, a_{11}=a_{33},\ a_{22}\in E^{(0)} \}$,  where $\underline{g}=(1_G, g_2, g_1) \in G^3$;

		\medskip
		\item[ ]
		$\mathcal{A}_{9}^{g_1,g_2,g_3} = E(M^{(e,g_1^0 ,(g_1g_2)^0, (g_1g_2)^{1},(g_1g_2g_3)^0)})$,  $g_i \in G$,   $1\leq i\leq 3$;

		\medskip
		\item[ ]
		$\mathcal{A}_{10}^{g_1,g_2,g_3} = E(N^{(e,g_1^0 ,g_1^1, (g_1g_2)^0,(g_1g_2g_3)^0)})$, $g_i \in G$, $1\leq i\leq 3$;

		\medskip
		\item[ ]
		$\mathcal{A}_{11}^{g_1,g_2} =E(P^{(e,g_1^0 ,g_1^1, (g_1g_2)^0)})$,  $g_i \in G$,    $1\leq i\leq 2$;
		
		\medskip
		\item[ ]
		$\mathcal{A}_{12}^{g_1,g_2} =E(P^{(e,g_1^0 ,g_1^1, (g_1g_2)^1)})$, $g_i \in G$,    $1\leq i\leq 2$.
	\end{enumerate}
	
	\bigskip
	Let's remark that, since $\mathcal{A}_2^p$ is a simple commutative $G$-graded algebra, $exp^{G,\delta}(\mathcal{A}_2^p)=\dim_F FC_{p}=p$. Also, it is clear that $exp^{G,\delta}(\mathcal{A}^{g,i}_1)=
	exp^{G,\delta}(\mathcal{A}_3)=exp^{G,\delta}(\mathcal{A}_4)=exp^{G,\delta}(\mathcal{A}_5^{i,j})=4$.

	Moreover,  if we regard the proper central polynomials constructed in \cite{BV}
	as $G$-polynomials, then we get  the following

	\begin{Lemma} \label{exponent Ai}
		Let $\mathcal{A}$ be a $G$-graded algebra belonging to  $\{\mathcal{A}_6^{g_1,g_2,g_3}, \mathcal{A}_7^{g_1,g_2,g_3, g_4}, \mathcal{A}_{8}^{g_1,g_2}, \mathcal{A}_{9}^{g_1,g_2,g_3}, \mathcal{A}_{10}^{g_1,g_2,g_3}, \mathcal{A}_{11}^{g_1,g_2}, \mathcal{A}_{12}^{g_1,g_2}\} $, then $exp^G(\mathcal{A})=exp^{G,\delta}(\mathcal{A})=3$.
	\end{Lemma}
	
	\bigskip
	
	Let's now consider the algebras $
	B_1=\left\{\left(
	\begin{array}{cc}
		0 & a \\
		0 & b  \\
	\end{array}
	\right) \,| \,  a,b \in E^{(0)}\right\}$ and
	$B_2=\left\{\left(
	\begin{array}{ccc}
		u & v &  a\\
		v & u& b \\
		0 & 0  & 0 \\
	\end{array}
	\right) \, | \, u,a \in E^{(0)}, v,b\in E^{(1)} \right\}.$
	Let $\langle f_1, \ldots, f_r \rangle_{T_G}$ denote the $T_G$-ideal of $F\langle X,G \rangle$ generated by the $G$-polynomials $ f_1, \ldots, f_r .$
	We have the following
	
	\begin{Lemma} \label{TidealB} Let $g\neq 1$. Then
		\hfill
		\begin{itemize}
			\item[1)]
			$Id^G(B_1^{(1,g)}) = \langle [x_1^1,x_2^1], x_1^gx_2^g, x_1^1x_2^g, x_1^h \rangle_{T_G}$,    $h \ne 1, g$;
			\item[2)]
			$Id^G(B_2^{(1,1,g)}) = \langle [x_1^1,x_2^1,x_3^1], x_1^gx_2^g, x_1^gx_2^1, x_1^h \rangle_{T_G},$     $h \ne 1, g$;
			\item[3)] $ Id^G(B_1^{(1,1)})=\langle x_1^1[x_2^1,x_3^1], x_1^h\rangle_{T_G}$,     $h \ne 1$;
			\item[4)] $ Id^G(B_2^{(1,1,1)})=\langle [x_1^1,x_2^1,x_3^1]x_4^1, x_1^h\rangle_{T_G}$,     $h \ne 1$.
		\end{itemize}
	\end{Lemma}
	
	\begin{proof} It is easy to see that $B_1^{(1,g)}$ satisfies the graded identities  $[x_1^1,x_2^1], x_1^gx_2^g, x_1^1x_2^g$, $x^h$, with $h\neq 1, g$.
		
		Write  $I=\langle [x_1^1,x_2^1], x_1^gx_2^g, x_1^1x_2^g, x^h\rangle_{T_G}$ and let $f=f(x_1^1,\ldots, x_{t_1}^1,x_1^{g_2},\ldots,x_{t_2}^{g_2}, \ldots,,x_1^{g_s},\ldots,x_{t_s}^{g_s})$
		be a multilinear $G$-polynomial in $Id^G(B_1^{(1,g)})$.
		We wish to show that, modulo $I$, $f$ is the zero polynomial.
		
		It is clear that, modulo $I$, we can write $f\equiv f_1(x_1^1,\ldots, x_{t_1}^1) +
		f_2(x_1^g, x_1^{1},\ldots,x_{t_1}^{1})$.  By the multihomogeneity of $T_G$-ideals, it follows that $f_1$
		and $f_2$ are both graded identities of $B_1^{(1,g)}$.
		Since $[x_1^1,x_2^1]\in I$, we obtain that $f_1 \equiv \alpha x_1^1 \cdots x_{t_1}^1$ for some $\alpha \in F.$  By substituting $x_1^1 = \cdots = x_{t_1}^1  = e_{22}$ we have $\alpha =0$ and, so, $f_1 \equiv 0$ modulo $I$.

		Now, modulo $I,$ $f_2 \equiv\beta x_1^g x_1^1 \cdots x_{t_1}^1,$ for some $\beta \in F$.  If we consider the following substitutions $x_1^g=e_{12}$ and $x_1^1 = \cdots = x_{t_1}^1  = e_{22}$, we get $f_2=\beta e_{12}$. Hence $\beta = 0,$  $f_2 \equiv 0$ modulo $I$, and we are done.
		
		Now let $I=\langle [x_1^1,x_2^1,x_3^1], x_1^gx_2^g, x_1^gx_2^1, x^h\rangle_{T_G}.$ It is easy to see that $I \subseteq  Id^G(B_2^{(1,1,g)}).$
		Let $f=f(x_1^1,\ldots, x_{t_1}^1,x_1^{g_2},\ldots, \\ x_{t_2}^{g_2}, \ldots,,x_1^{g_s},\ldots,x_{t_s}^{g_s})$
		be a multilinear polynomial in $Id^G(B_2^{(1,1,g)})$.
		Write $f\equiv f_1(x_1^1,\ldots, x_{t_1}^1) +
		f_2( x_1^{1},\ldots,x_{t_1}^{1},x_1^g)$ where
		$$
		f_1(x^1_1, \ldots , x^1_{t_1})\equiv\sum \beta_{\underline i,\underline j}x^1_{i_1}\cdots x^1_{i_{t_1-2m}}[x^1_{j_1},x^1_{j_2}]\cdots [x^1_{j_{2m-1}},x^1_{j_{2m}}]
		$$
		and
		$$
		f_2(x^1_1, \ldots , x^1_{t_1},x_1^g)\equiv\sum \gamma_{\underline r,\underline s}x^1_{r_1}\cdots x^1_{r_{t_1-2q}}[x^1_{s_1},x^1_{s_2}]\cdots [x^1_{s_{2q}},x^1_{s_{2q}}] x_1^g,
		$$
		with $\beta_{\underline i,\underline j}, \gamma_{\underline r,\underline s} \in F$, $\underline i =(i_1, \ldots, i_{t_1-2m})$,
		$i_1<\cdots < i_{t_1-2m}$, $\underline j=(j_1, \ldots,  j_{2m})$,
		$j_1<\cdots < j_{2m}$,
		$\underline r =(r_1, \ldots, r_{t_1-2q})$,
		$r_1<\cdots < r_{t_1-2q}$, $\underline s=(s_1, \ldots,  s_{2q})$ and
		$s_1<\cdots < s_{2q}$.
		
		By the multihomogeneity of $T_G$-ideals, it follows that $f_1$
		and $f_2$ are both graded identities of $B_2^{(1,1,g)}$. Suppose that $f_1 \not \equiv 0$ modulo $I$.
		
		We choose the minimal $m$ such that $\beta_{\underline i,\underline j} \neq 0$ and fix one of the corresponding $(t_1-2m, 2m)$-tuples \\ $(i_1,\ldots ,i_{t_1-2m}; j_1, \ldots, j_{2m})$.
		We make the substitution $x_{i_1}=\cdots = x_{i_{t_1-2m}}= e_{11}+e_{22}$, $x_{j_t}=a_t(e_{12}+e_{21})$, for all $t=1, \ldots ,  2m$,   where  the $a_t$'s are distinct generators of $E^{(1)}$.
		Clearly we obtain, for a suitable coefficient $\beta_{\underline i,\underline j} \in F$,
		
		$$
		\beta_{\underline i,\underline j} [a_1, a_2]\cdots  [a_{2m-1}, a_{2m}] (e_{11}+ e_{22})=2^m\beta_{\underline i,\underline j} a_1\cdots a_{2m}(e_{11}+ e_{22})\neq 0.
		$$
		
		\noindent
		Since $[e_{11}+e_{22},e_{12}+e_{21}]=0$, the minimality of $m$ gives that all other summand of $f_1$ vanish and we get a contradiction.
		Similar considerations hold for $f_2$.
		Thus $f \equiv 0$ modulo $I$, and we are done.
		
		The proof of $3)$ and $4)$ follows by \cite[Lemma 3.2]{BV}.
	\end{proof}
	
	\bigskip
	
	Next consider the algebras
	$C_1=\left\{\left(
	\begin{array}{ccc}
		0 & a & b\\
		0 & u & v \\
		0 & v & u \\
	\end{array}
	\right) \, | \,  u,a \in E^{(0)}, v,b\in E^{(1)} \right\}$ and $
	C_2=\left\{\left(
	\begin{array}{cc}
		a & b \\
		0 & 0  \\
	\end{array}
	\right) \, | \,  a,b \in E^{(0)}\right\}$.

    In a similar vein as in the proof of the previous lemma  we obtain the following
	
	\bigskip
	
	\begin{Lemma} \label{TidealC} If $ g \ne 1,$ then
		\hfill
		\begin{itemize}
			\item[1)]
			$Id^G(C_1^{(1,g,g)}) = \langle [x_1^1,x_2^1,x_3^1], x_1^gx_2^g, x_1^1x_2^g, x_1^h \rangle_{T_G},$ $h\neq 1,g$;
			\item[2)]
			$Id^G(C_2^{(1,g)}) = \langle [x_1^1,x_2^1], x_1^gx_2^g, x_1^gx_2^1, x_1^h \rangle_{T_G}$, $h\neq 1,g$;
			\item[3)] $Id^G(C_1^{(1,1,1)})=\langle x_1^1 [x_2^1,x_3^1,x_4^1], x_1^h\rangle_{T_G}$, $h\neq 1$;
			\item[4)] $Id^G(C_2^{(1,1)})=\langle [x^1_1,x_2^1]x_3^1, x_1^h\rangle_{T_G}$, $h\neq 1.$
		\end{itemize}
	\end{Lemma}
	
	\bigskip
	
	Recall that the free supercommutative algebra $S = F[U, V]$ is the algebra with $1$ generated by two countable sets
	$U =\{u_{j}^i\}$ and $V =\{v_{j}^i\}$ over $F$, subject to the condition that the elements of $U$ are central and the elements of $V$ anticommute among them.
	The algebra $S$ has a natural $ \mathbb{Z}_2$-grading,  $S=S^{(0)} \oplus S^{(1)}$, by requiring that the variables of $U$ are of homogeneous degree zero and  those of $V$ are of homegeneous degree one. We shall call the elements of $U$ commutative o polynomial variables and the elements of $V$ Grassmann variables. Notice that the Grassmann algebra $E$ can be viewed as a $\mathbb{Z}_2$-graded subalgebra of $S$ if we identify the generating elements $e_k$ of $E$ with the elements of $V$. Hence $S \cong E \otimes F[U], S^{(0)}\cong E^{(0)}\otimes F[U]$ and $S^{(1)}\cong E^{(1)} \otimes F[U].$

	If $B=\bigoplus_{(g,i) \in G\times \mathbb{Z}_2} B^{(g,i)}$ is a finite dimensional $G \times \mathbb{Z}_2$-graded algebra over an algebraically closed field $F,$ one can define the superenvelope of $B$ to be  the algebra $S(B) = (S^{(0)} \otimes B^{(0)}) + (S^{(1)} \otimes B^{(1)})$ where $B^{(0)}=  \bigoplus_{g \in G} B^{(g,0)}$ and $B^{(1)}=  \bigoplus_{g \in G} B^{(g,1)}.$  Clearly $ S(B)$ has an induced $G$-grading where $S(B)^g \cong E(B)^g \otimes F[U],$ with $g \in G$ and $Id^G(S(B)) = Id^G(E(B)).$
	
	Let $ \mathcal B$ denote a basis of homogeneous elements with respect to the $G \times \mathbb{Z}_2$-grading of $B$, write $\mathcal B=\{a_1, \ldots, a_r, b_1, \ldots, b_t\}$ where $\{a_1, \ldots, a_r\}$ is a basis of $B^{(0)}$ and $\{b_1, \ldots, b_t\}$ is a basis of $B^{(1)}.$ For a fixed $n \ge 1,$  we choose $nr$ variables $u_{j}^i \in U$, $i =1, \ldots, n, j=1, \ldots, r$, and $nt$ variables $ v_{j}^i \in V, i =1, \ldots, n, j=1, \ldots, t.$ For $i =1, \ldots, n$ and  $g \in G,$ we define the generic elements
	$$
	Z^{i,g} = \sum_{r_j} u_{r_j}^i\otimes a_{r_j} + \sum_{t_j} v_{t_j}^i\otimes b_{t_j},
	$$
	where the first sum runs over all $r_j$ such that $a_{r_j}$ is of homogeneous degree $(g, 0)$ and the second one runs over all $t_j$ such that $b_{t_j}$ is of homogeneous degree $(g, 1).$
	We denote by $\mathcal{H}$ the $G$-graded subalgebra of $S(B)$ generated by the generic elements $Z^{i,g} , i =1, \ldots, n$, and $g\in G$.
	We have the following basic property (see \cite{LMR}).
	
	\begin{Proposition} \label{free}The algebra $\mathcal{H}$ is isomorphic to the relatively free $G$-graded algebra for $E(B)$ in $ns$ graded generators.
	\end{Proposition}
	
	Now we are able to prove the following.
	\begin{Proposition} \label{A10,A11} Let $g_1,g_2,g_3 \in G$. Then
		\hfill
		\begin{itemize}
			\item[1)] $Id^G(\mathcal{A}_{9}^{g_1,g_2,g_3})=$ $Id^G(B_1^{(1,g_1)})Id^G(B_2^{(1,1,g_3)})$
			
			\item[2)] $Id^G(\mathcal{A}_{10}^{g_1,g_2,g_3})=Id^G(C_1^{(1,g_1, g_1)})Id^G(C_2^{(1,g_3)})$
		\end{itemize}
	\end{Proposition}
	\begin{proof}
		We write
		$$
		\mathcal{A}_{9}^{g_1,g_2,g_3}=\left(
		\begin{array}{cc}
			B_1 & J \\
			0 & B_2 \\
		\end{array}
		\right)^{(1,g_1, g_1g_2, g_1g_2, g_1g_2g_3)}
		$$
		
		\medskip
		\noindent
		where
        $B_1$ and $B_2$ are the algebras defined before Lemma \ref{TidealB} and
		$J=\left(
		\begin{array}{ccc}
			E^{(0)} & E^{(1)} & E^{(0)} \\
			E^{(0)} & E^{(1)} & E^{(0)} \\
		\end{array}
		\right)$.
			
		We shall prove that $Id^G(\mathcal{A}_{9}^{g_1,g_2,g_3})=Id^G(B_1^{(1,g_1)})Id^G(B_2^{(1,1,g_3)})$ by making use of Proposition \ref{free}.

        Recall that $\mathcal{A}_{9}^{g_1,g_2,g_3}=E(B)$, where $B=M^{(e,g_1^0 ,(g_1g_2)^0, (g_1g_2)^{1},(g_1g_2g_3)^0)}$
		with $g_i \in G$,
		$1\leq i\leq 3.$
		Notice that $B=B^{(0)}\oplus B^{(1)}$  where  dim$_FB^{(0)}=8$ and dim$_FB^{(1)}=4$.
		
		Let $\{a_1, \ldots , a_8\}$ be a $G$-graded basis of $B^{(0)}$ and let $\{b_1, \ldots , b_4\}$ be a $G$-graded basis of $B^{(1)}$.
		We choose $8n$ homogeneous polynomial variables $u^i_{j}$, $1\le i\le n$, $1\le j\le 8$, and $4n$ homogeneous Grassmann variables  $v^i_{j}$, $1\le i\le n$, $1\le j\le 4$.
		Denote by $\mathcal{H}$ the subalgebra generated by the elements
		$$
		Z^{i,g}=\sum u^i_{r_j} \otimes a_{r_j} + \sum v^i_{t_j}\otimes b_{t_j} \in  F[U, V]\otimes B,
		$$
		for $i = 1, \ldots, n,$ $ g\in G, $ where the first sum runs over all $r_j$ such that $a_{r_j}$ is of homogeneous degree $(g, 0)$ and the
		second one runs over all $t_j$ such that $b_{t_j}$ is of homogeneous degree $(g, 1).$
		
		Since the basis of $B$ is made of matrix units, we shall use a double index for the commutative and Grassmann variables.
		Let us consider, for $i=1, \ldots , n$, the following three infinite disjoint sets of commutative and Grassmann variables
		$$
		X^i=\{u^i_{1,2}, u^i_{2,2}\}, \,\, Y^i=\{u^i_{3,3}, u^i_{3,5}, v^i_{3,4}, v^i_{4,5}\}, \,\,
		U^i=\{u^i_{1,3}, u^i_{1,5}, u^i_{2,3}, u^i_{2,5}, v^i_{1,4}, v^i_{2,4}\},
		$$
		
		\noindent and the following matrices
		$$
		X^{i,g_1}=u^i_{1,2} \otimes e_{12}, \,\, X^{i,1} =u^i_{2,2}\otimes e_{22},
		$$
		$$
		Y^{i,1}=u^i_{3,3}\otimes (e_{33}+e_{44})+ v^i_{3,4}\otimes (e_{34}+e_{43}), \, \, Y^{i,g_3} =u^i_{3,5}\otimes e_{35}+v^i_{4,5}\otimes e_{45},
		$$
		$$
		U^{i,g_1g_2}=u^i_{1,3}\otimes e_{13}+ v^i_{1,4}\otimes e_{14},
		U^{i,g_1g_2g_3} =u^i_{1,5}\otimes e_{15},$$
        $$
		U^{i,g_2}=u^i_{2,3}\otimes e_{23}+v^i_{2,4}\otimes e_{24}, \,\,
		U^{i,g_2g_3}=u^i_{2,5}\otimes e_{25}.
		$$
		
		Then, by the previous Proposition \ref{free}, the generic matrices $X^{i,g_1} , X^{i,1}$ generate the relatively free $G$-graded algebra $\tilde{B}_1$ of rank $2n$ of
		the variety $var^G(B_1^{(i,g_1)})$ and $Y^{i,1} , Y^{i,g_3}$ generate the relatively free $G$-graded algebra $\tilde{B}_2$ of rank $2n$ of
		the variety $var^G(B_2^{(1,1,g_3)})$.
		
		If we prove that, for $i =1 \ldots, n,$ $U^{i,g_1g_2}, U^{i,g_1g_2g_3} ,U^{i,g_2}, U^{i,g_1g_2}, U^{i,g_2g_3}$ generate a $G$-graded free
		$(\tilde{B}_1, \tilde{B}_2$)-bimodule then, by  \cite[Corollary 3.2]{DVLS},
		we obtain that the set  $\{ Z^{i,g}=X^{i,g}+ Y^{i,g}+ U^{i,g} | g \in G, \, 1\le i\le n \}$ generates a relatively free $G$-graded
		algebra of the variety corresponding to  the product of the $T_G$-ideals $Id^G(B_1^{(1,g_1)})Id^G(B_2^{(1,1,g_3)})$.
		This will complete the proof of the proposition.

		Hence let us show that, for $i = 1, \ldots,n,$ $U^{1}_i=U^{i,g_1g_2}, U^{2}_i=U^{i,g_1g_2g_3} , U^{3}_i=U^{i,g_2}, U^{4}_i=U^{i,g_2g_3}$ generate a $G$-graded free
		$(\tilde{B}_1, \tilde{B}_2)$-bimodule.
		Let $a^j_{l,i},$ $1\le l\le 4$,  be  non-zero elements  of $\tilde{B}_1$
		and let  $b_{l,1},b_{l,2}, \ldots $, $1\le l\le 4$,   be linearly independent elements of $\tilde{B}_2$.
		We want to prove that
		$$
		\sum_i\sum_{j=1}^na^j_{1,i}U^{1}_jb_{1,i}+\cdots +\sum_i\sum_{j=1}^na^j_{4,i}U^{4}_jb_{4,i}\neq 0.
		$$
		Since $U^1_j, \ldots , U^4_j$, for $j = 1, \ldots,n,$ depend on disjoint sets of variables it is enough to check that
		$$
		a^1_{3,1}U_1^3b_{3,1}+\cdots +a^1_{3,t}U_1^3b_{3,t}\neq 0,
		$$
		for all $t\geq 1$.
		Since $a^1_{3,1}\neq 0$, it has a non zero value in $B_1^{(1,g_1)}$ under some specialization of the variables in $X^1$. This value is of the type
		$\lambda_{1,2}^1\otimes e_{12}+\lambda_{2,2}^1\otimes e_{22}\neq 0,$
		with $\lambda_{1,2}^1, \lambda_{2,2}^1 \in E^{(0)}$.
		Suppose now that
		$
		a^1_{3,1}U_1^3b_{3,1}+\cdots +a^1_{3,t}U_1^3b_{3,t}= 0.
		$
		Then
		$$
		\sum_{i=1}^ta^1_{3,i}(u^i_{2,3}\otimes e_{23}+  v^i_{2,4}\otimes e_{24})b_{3,i}=0
		$$
		and, so,
		$$
		\sum_{i=1}^t (\lambda_{1,2}^{i}\otimes e_{12}+\lambda_{2,2}^{i}\otimes e_{22})(u^i_{2,3}\otimes e_{23}+  v^i_{2,4}\otimes e_{24}) b_{3,i} =0.
		$$
		Assume for instance that $\lambda_{1,2}^{1}\neq 0$.
		If we consider the following specialization
		$
		u^i_{2,3}\rightarrow 1, \,\,  v^i_{2,4}\rightarrow 0,
		$
		then
		$$
		\sum_{i=1}^t (\lambda_{1,2}^{i}\otimes e_{13}+\lambda_{2,2}^{i}\otimes e_{23})b_{3,i} =0
		$$
		and in particular,
		$$
		(1\otimes e_{13})(\sum_{i=1}^t \lambda_{1,2}^{i} b_{3,i})  =0.
		$$
		Let $f=\sum_{i=1}^t \lambda_{1,2}^{i} b_{3,i}.$ We observe that $f$ is not an identity for $B_2^{(1,1,g_3)}$ but for any specialization of the variables in $Y$,  i.e. for any $G$-graded homomorphism $\phi: \tilde{B}_2\rightarrow B_2^{(1,1,g_3)}$, we have $e_{13}\phi(f)=0.$ It follows that in $\phi(f)$
		all elements of the first row are zero.
        Thus $\phi(f)\in \langle E^{(1)}e_{45} \rangle$. Notice that in a non-zero evaluation of $f$ only one variable in each monomial of $f$ can be evaluated in $E^{(0)}e_{35}$  or $E^{(1)}e_{45}$; all the other variables must be evaluated in $E^{(0)}(e_{33}+e_{44})+E^{(1)}(e_{34}+e_{43}).$ Then in each monomial of $f$ at most one variable of degree $g_3$ can appear.
        Write $f=f_0+f_1x_i^{g_3}$, where a variable $x_j^{g_3}$ never appears as rightmost variable in $f_0$.
        Now let $\phi(x_i^{g_3})=d e_{45}$, with $d \in E^{(1)}.$ Hence $0 \neq \phi(f)= \phi(f_0)+\phi(f_1x_i^{g_3})=\phi(f_1)\phi(x_i^{g_3})=\phi(f_1)de_{45}$
        $=(a (e_{33}+e_{44})+  b(e_{34}+e_{43}))de_{45}= a d e_{45}+ bd e_{35} \in \langle E^{(1)}e_{45} \rangle$, with $a \in E^{(0)}$ and $b\in E^{(1)}$. Thus $b =0.$
         Now if we consider the evaluation $\phi'$ that differs from $\phi$ only for the value of  $x_i^{g_3}$ that is now evaluated in  $d'e_{35}$ with $d'\in E^{(0)}$, then we get
        $\phi'(f)= \phi'(f_0)+\phi'(f_1x_i^{g_3})=\phi(f_1)d'e_{35}$
        $=( a (e_{33}+e_{44})+ b(e_{34}+e_{43}))d'e_{35}= a d' e_{35}+ bd' e_{45} \in \langle E^{(1)}e_{45} \rangle$. This says that $a=0$. Thus $a=b=0$, and this contradicts $\phi(f)\ne 0$.

		The contradiction just obtained proves that $Id^G(\mathcal{A}_{9}^{g_1,g_2,g_3})=Id^G(B_1^{(1,g_1)})Id^G(B_2^{(1,1,g_3)}).$
		
		The second part is proved similarly.
	\end{proof}
	\medskip
	By  Proposition \ref{A10,A11} and its proof we have the following.
	\medskip
	
	\begin{Remark} \label{rem1}
	For the following $G$-graded algebras

		\bigskip
		
		$\mathcal{A}_{9,1}^{g_1,g_2,g_3} =   E( M^{(e,g_1^0 ,(g_1g_2)^0, (g_1g_2)^{1},(g_1g_2g_3)^1)}),$		

        \medskip
		$\mathcal{A}_{9,2}^{g_1,g_2,g_3} =   E(M^{(e,g_1^1 ,(g_1g_2)^1, (g_1g_2)^{0},(g_1g_2g_3)^1}),$	

        \medskip
		$\mathcal{A}_{9,3}^{g_1,g_2,g_3} =   E(M^{(e,g_1^1 ,(g_1g_2)^1, (g_1g_2)^{0},(g_1g_2g_3)^0}),$

		\medskip
		$\mathcal{A}_{10,1}^{g_1,g_2,g_3} =  E( N^{(e,g_1^1 ,g_1^0, (g_1g_2)^1,(g_1g_2g_3)^1}),$
		
		\medskip
		$\mathcal{A}_{10,2}^{g_1,g_2,g_3}  =     E( N^{(e,g_1^0 ,g_1^1, (g_1g_2)^0,(g_1g_2g_3)^1)}),$
		
		\medskip
		$\mathcal{A}_{10,3}^{g_1,g_2,g_3}  =   E( N^{(e,g_1^1 ,g_1^0, (g_1g_2)^1,(g_1g_2g_3)^0)}),$

		\medskip
		\noindent with  $g_i \in G$,  $1\leq i\leq 3$,
        we have
		$$
		Id^G(\mathcal{A}_{9,1}^{g_1,g_2,g_3} )=Id^G(\mathcal{A}_{9,2}^{g_1,g_2,g_3} )=Id(\mathcal{A}_{9,3}^{g_1,g_2,g_3} )=Id^G(\mathcal{A}_{9}^{g_1,g_2,g_3} )
		$$
		and
		$$
		Id^G(\mathcal{A}_{10,1}^{g_1,g_2,g_3} )=Id^G(\mathcal{A}_{10,2}^{g_1,g_2,g_3} )=Id^G(\mathcal{A}_{10,3}^{g_1,g_2,g_3} )=Id(\mathcal{A}_{10}^{g_1,g_2,g_3} ).
		$$
	\end{Remark}
	\bigskip
	
	\section{Constructing $G$-graded algebras of $(G,\delta)$-exponent greater than 2}

	Throughout this section $B=\oplus_{(g,i) \in G \times \mathbb{Z}_2} B^{(g,i)}$ will be a finite dimensional $G \times \mathbb{Z}_2$-graded algebra  over an algebraically closed field $F$ of characteristic zero.
	
	We start with the following
	
	\begin{Lemma} \label{lemmaA_7}
		If $B^{(1,0)}$ contains three orthogonal idempotents $e_1,e_2,e_3$ such that
		$e_1j_1e_2j_2e_3j_3e_1\ne 0$, for some $j_1,j_2,j_3\in J(B)$, then there exist  $g_1,g_2,g_3 \in G$  and a $G \times \mathbb{Z}_2$-graded subalgebra $\bar B$ of $B$  such that $Id^G(E(\bar B))\subseteq Id^G(\mathcal{A}_6^{g_1,g_2,g_3}).$
	\end{Lemma}
	
	\begin{proof}
		By linearity we may assume that the elements $j_1,  j_2, j_3$ are homogeneous of degree $h_1=(g_1, i_1), h_2=(g_2, i_2)$ and $h_3=(g_3, i_3)$, respectively.
		Let $\bar B$ be the $ G \times \mathbb{Z}_2$-graded subalgebra of $B$ generated by the homogeneous elements
		$$
		e_1,  e_2,  e_3, e_1j_1e_2, e_2j_2e_3, e_3j_3e_1.
		$$
		We consider  $UT_4(F)^{\underline{h}}$ the algebra of $4 \times 4$ upper triangular matrices  with elementary $G\times\mathbb{Z}_2$-grading induced by $\underline{h}=(e, h_1, h_1h_2, h_1h_2h_3).$ We construct a homomorphism of $G \times \mathbb{Z}_2$-graded algebras $\varphi :\bar B\rightarrow UT_4(F)^{\underline{h}}$  by setting $\varphi(e_1)=e_{11}+e_{44}$, $\varphi(e_2)=e_{22}$, $\varphi(e_3)=e_{33}$, $\varphi(e_1j_1e_2)=e_{12}$, $\varphi(e_2j_2e_3)=e_{23}$ and $\varphi(e_3j_3e_1)=e_{34}$. Then $I=Ker\varphi$ is the ideal of $\bar B$ generated by the element
		$$
		e_3j_3e_1j_1e_2
		$$
		and $\bar B/I\cong \varphi(\bar B)$ is a subalgebra of $UT_4(F)^{\underline{h}}$.
		Now, it can be checked that the elements
		$$
		e_1,  e_2,  e_3, e_1j_1e_2, e_2j_2e_3, e_3j_3e_1, e_1j_1e_2j_2e_3, e_2j_2e_3j_3e_1, e_1j_1e_2j_2e_3j_3e_1
		$$
		are linearly independent and form a basis of $\bar B$ mod $I$.
		By abuse of notation we identify these representatives with the corresponding cosets.
		
		Let $a_1, a_2, a_3$ be distinct homogeneous elements of $E$ having homogeneous degree  $i_1, i_2, i_3$, respectively. Then the elements
		$$
		1\otimes e_1,\  1\otimes e_2,\  1\otimes e_3,\ a_1\otimes e_1j_1e_2,\ a_2\otimes e_2j_2e_3,\ a_3\otimes e_3j_3e_1,
		$$
		$$
		a_1a_2\otimes e_1j_1e_2j_2e_3,\ a_2a_3\otimes e_2j_2e_3j_3e_1,\  a_1a_2a_3 \otimes e_1j_1e_2j_2e_3j_3e_1,
		$$
		form a basis of a $G$-graded subalgebra $\mathcal C$ of $E(\bar B/I)$ isomorphic to $\mathcal{A}_6^{g_1,g_2,g_3}.$
		
		We obtain that  $Id^G(\mathcal{A}_6^{g_1,g_2,g_3})= Id^G(\mathcal C) \supseteq Id^G(E(\bar B/I)) \supseteq Id^G(E(\bar B))$, and the proof is complete.
	\end{proof}
	
	\begin{Lemma} \label{lemmaA_8}
		If $B^{(1,0)}$ contains three orthogonal idempotents $e_1,e_2,e_3$ such that
		$j_1e_1j_2e_2j_3e_3j_4\ne 0$,  $j_i \in J(B)$,    then there exist  $g_1,g_2,g_3, g_4 \in G$  and  a $G\times \mathbb{Z}_2$-graded subalgebra $\bar B$ of $B$  such that $Id^G(E(\bar B))\subseteq Id^G(\mathcal{A}_7^{g_1,g_2,g_3, g_4}).$
	\end{Lemma}
	
	\begin{proof}
		As in the previous lemma we assume that the elements  $j_1,j_2,j_3, j_4$ are homogeneous of degree $h_t=(g_t,i_t)$, for $t=1,\ldots , 4,$ respectively.
		Let $\bar B$ be the $G \times \mathbb{Z}_2$-graded subalgebra of $B$ generated by the homogeneous elements
		$$
		e_1,  e_2,  e_3, j_1e_1, e_1j_2e_2, e_2j_3e_3, e_3j_4.
		$$
		Next we consider $UT_5(F)^{\underline{h}}$,  the algebra $UT_5(F)$ with elementary $G\times\mathbb{Z}_2$-grading induced by $\underline{h}=(e, h_1, h_1h_2, h_1h_2h_3, \\ h_1h_2h_3h_4)$ and we build a homomorphism of $G \times \mathbb{Z}_2$-graded algebras $\psi :\bar B\rightarrow UT_5(F)^{\underline{h}}$ defined by setting $\psi(e_1)=e_{22}$, $\psi(e_2)=e_{33}$, $\psi(e_3)=e_{44}$, $\psi(j_1e_1)=e_{12}$, $\psi(e_1j_2e_2)=e_{23}$, $\psi(e_2j_3e_3)=e_{34}$ and $\psi(e_3j_4)=e_{45}$. Then $I=Ker \psi$ is the homogeneous ideal of $\bar B$ generated by the elements
		$$
		e_3j_4e_1, \, e_3j_4e_2, \, e_3j_4e_3, \, e_3j_4j_1e_1, \, e_3j_1e_1, \,  e_2j_1e_1, \,
		e_1j_1e_1
		$$
		and $\bar B/I\cong \psi(\bar B)$ a subalgebra of $UT_5(F)^{\underline{h}}$.
		The elements
		$$
		e_1,  e_2,  e_3, \
		j_1e_1,\
		e_1j_2e_2,\
		e_2j_3e_3,\
		e_3j_4,\ j_1e_1j_2e_2,\ e_1j_2e_2j_3e_3,\
		$$
		$$
		e_2j_3e_3j_4,\ j_1e_1j_2e_2j_3e_3,\
		e_1j_2e_2j_3e_3j_4,\
		j_1e_1j_2e_2j_3e_3j_4
		$$
		are linearly independent  and form a basis of $\bar B$ mod $I$.
		
		By abuse of notation we identify these representatives with the corresponding cosets.
		Then, for distinct $a_t\in  E^{(0)}\cup E^{(1)}$ having homogeneous  $\mathbb{Z}_2$-degree equal to $i_t$, with $t=1, \ldots , 4$, the elements
		$$
		1\otimes e_1,  1\otimes e_2,  1\otimes e_3,\ a_1\otimes  j_1e_1,\
		a_2 \otimes e_1j_2e_2,\
		a_3 \otimes e_2j_3e_3,\
		a_4 \otimes e_3j_4,\
		$$
		$$
		a_1a_2  \otimes j_1e_1j_2e_2,\ a_2a_3\otimes e_1j_2e_2j_3e_3,\
		a_3a_4\otimes e_2j_3e_3j_4,\
		$$$$
		a_1a_2a_3 \otimes j_1e_1j_2e_2j_3e_3,\
		a_2a_3a_4 \otimes e_1j_2e_2j_3e_3j_4,\
		a_1a_2a_3a_4 \otimes j_1e_1j_2e_2j_3e_3j_4,
		$$
		
		\smallskip
		\noindent are a basis of a subalgebra
		$\mathcal C$ of $E(\bar B/I)$ isomorphic to $\mathcal{A}_7^{g_1,g_2,g_3,g_4}$
		
		It follows that $Id^G(\mathcal{A}_7^{g_1,g_2,g_3, g_4})=Id^G(\mathcal C) \supseteq Id^G(E(\bar B/I)) \supseteq Id^G(E(\bar B))$, a desired conclusion.
	\end{proof}
	
	\begin{Lemma} \label{lemmaA_9}
		If $B^{(1,0)}$ contains two orthogonal idempotents $e_1,e_2$, with $e_2\in F\oplus cF$, $c^2=1$, such that
		$e_2j_2e_1j_1e_2\ne 0$, for some $j_1,j_2\in J(B)$,   then there exist $g_1,g_2 \in G$ and a $G \times \mathbb{Z}_2$-graded subalgebra $\bar B$ of $B$  such that $Id^G(E(\bar B))\subseteq Id^G(\mathcal{A}_8^{g_1,g_2})$.
	\end{Lemma}
	\begin{proof}
		Let $j_1,  j_2 \in J(B)$ be  homogeneous of degree $h=(g_1, i_1)$ and $ k =(g_2, i_2)$, respectively and let $\bar B$ be the $G \times \mathbb{Z}_2$-graded subalgebra of $B$ generated by the homogeneous elements
		$$
		e_1,  e_2, ce_2,   e_1j_1e_2, e_2j_2e_1.
		$$
		
		Consider the homomorphism of $G \times \mathbb{Z}_2$-graded algebras $\varphi : \bar B\rightarrow UT_3(F\oplus cF)^{(e, k, kh)}$  defined by setting $\varphi(e_1)=e_{22}$, $\varphi(e_2)=e_{11}+e_{33}$, $\varphi(ce_2)=c(e_{11}+e_{33})$, $\varphi(e_1j_1e_2)=\alpha e_{23}$ and $\varphi(e_2j_2e_1)=\beta e_{12}$,   where $\alpha, \beta \in \{1, c\}$ according to  $i_1$ and $i_2$.
		Then $I=Ker\varphi$ is the ideal of $ \bar B$ generated by
		$
		e_1j_1e_2j_2e_1, e_1j_1ce_2j_2e_1
		$
		and $ \bar B/I\cong \varphi( \bar B)=\bar A$, where
		$$\bar A=\left \{\begin{pmatrix}
			a & x & y  \\
			0& b & z  \\
			0 & 0 & a
		\end{pmatrix} \, | \, b\in F, a,x,y,z\in F\oplus cF \right \}
		$$
		\noindent is the subalgebra of $UT_3(F\oplus cF)$ with elementary  $G \times \mathbb{Z}_2$ grading induced by $(e, k, kh).$
		The following elements
		$$
		e_1,  e_2, ce_2,   e_1j_1e_2, e_2j_2e_1, e_1j_1ce_2, ce_2j_2e_1, e_2j_2e_1j_1e_2, ce_2j_2e_1j_1e_2
		$$
		are linearly independent and form a basis of $\bar B$ mod $I.$ By identifying these representatives with the corresponding cosets
		we obtain that  the elements
		$$
		1\otimes e_1,  1\otimes e_2,  1\otimes ce_2, 1 \otimes e_1j_1e_2, 1 \otimes  e_2j_2e_1,$$ $$ 1 \otimes e_1j_1ce_2,  1 \otimes ce_2j_2e_1,  1 \otimes e_2j_2e_1j_1e_2,  1 \otimes ce_2j_2e_1j_1e_2
		$$
		form a basis of $E(\bar B/I)\cong E(\bar A)=\mathcal{A}_8^{g_1,g_2}$ and so
		$Id^G(\mathcal{A}_8^{g_1,g_2}) \supseteq Id^G(E(\bar B))$.
	\end{proof}
	
	\begin{Lemma} \label{lemmaA_{10}}
		Let $B^{(1,0)}$ contain two orthogonal idempotents $e_1,e_2$, with $e_2\in F\oplus cF$, $c^2=1$, such that
		$j_1e_1j_2e_2j_3\ne 0$, for some  $j_1,j_2,j_3\in J(B)$. Then there exist $g_1,g_2,g_3 \in G$ and a $G \times\mathbb{Z}_2$-graded subalgebra $\bar B$ of $B$  such that $Id^G(E(\bar B))\subseteq Id^G(\mathcal{A}_{9}^{g_1,g_2,g_3})$ .
	\end{Lemma}

	\begin{proof}
		Start by  assuming that the elements  $j_t$ have grading $(g_t,0)=g_t^0$ for $t=1,2,3$, and let $\bar B$ be the $G \times \mathbb{Z}_2$-graded subalgebra of $B$ generated by the homogeneous elements
		$$
		e_1,  e_2, ce_2,   j_1e_1, e_1j_2e_2, e_2j_3.
		$$
		such that $j_1e_1j_2e_2j_3\ne 0$.
		We consider the homomorphism of $G \times \mathbb{Z}_2$-graded algebras $$\varphi : \bar B\rightarrow UT(2,3)^{(e,g_1^0,(g_1g_2)^0,(g_1g_2)^{1},(g_1g_2g_3)^0)}$$  by setting $\varphi(e_1)=e_{22}$, $\varphi(e_2)=e_{33}+e_{44}$, $\varphi(ce_2)=e_{34}+e_{43}$, $\varphi(j_1e_1)=e_{12}$, $\varphi(e_1j_2e_2)=e_{23}$   and $\varphi(e_2j_3)=e_{35}$.
		Let $I=Ker\varphi$, then $\bar B/I\cong \varphi(\bar B)=M$ where

		$$
		M= M^{(0)} \oplus M^{(1)} = \left\{\begin{pmatrix}
			0 & a & b & d & e  \\
			0 & f & g & h & l \\
			0 & 0 & u & v & m \\
			0 & 0 & v & u & n \\
			0 & 0 & 0 & 0& 0
		\end{pmatrix}\right\} =  \left\{\begin{pmatrix}
			0 & a & b & 0 & e  \\
			0 & f & g & 0 & l \\
			0 & 0 & u & 0 & m \\
			0 & 0 & 0 & u & 0 \\
			0 & 0 & 0 & 0& 0
		\end{pmatrix}\right\}   \oplus \left\{ \begin{pmatrix}
			0 & 0 & 0 & d & 0  \\
			0 & 0 & 0 & h & 0 \\
			0 & 0 & 0 & v & 0 \\
			0 & 0 & v & 0 & n \\
			0 & 0 & 0 & 0& 0
		\end{pmatrix}\right\} .
		$$
		
		\medskip

		The following elements
		$$
		e_1,  e_2,\ ce_2,\   j_1e_1,\ e_1j_2e_2,\ e_2j_3,\ j_1e_1j_2e_2,
		$$
		$$
		j_1e_1j_2ce_2,\ j_1e_1j_2e_2j_3,\ e_1j_2ce_2,\ e_1j_2e_2j_3,\
		ce_2j_3
		$$
		
		\noindent \smallskip
		are linearly independent mod $I$ and form a basis of $\bar B/I$.
		By the same procedure of the previous lemmas it follows that
		$\mathcal{A}_{9}^{g_1,g_2,g_3} = E(M^{(e,g_1^0,(g_1g_2)^0,(g_1g_2)^{1},(g_1g_2g_3)^0)})$
		and so  $Id^G(\mathcal{A}_{9}^{g_1,g_2,g_3}) \supseteq Id^G(E(\bar B))$.
		\smallskip
		Similarly we obtain the same result when $j_1 \in J^{(0)}$ and $j_2, j_3 \in J^{(1)}.$
		Moreover if $j_1 \in J^{(0)}$ and $j_2, j_3$ are of different homogeneous degree with  respect to the  induced $\mathbb{Z}_2$-grading,
		it follows that  $Id^G(\mathcal{A}_{9,1}^{g_1,g_2,g_3})\supseteq  Id^G(E(\bar B))$.

		In case $j_1 \in J^{(1)}$ then we get that either $Id^G(\mathcal{A}_{9,2}^{g_1,g_2,g_3}) \supseteq Id^G(E(\bar B))$
		or $Id^G(\mathcal{A}_{9,3}^{g_1,g_2,g_3}) \supseteq Id^G(E(\bar B))$, according to whether $j_2, j_3$ are of the same or of different homogeneous degree with respect to the  induced $\mathbb{Z}_2$-grading.

		Since by Remark \ref{rem1}, $Id^G(\mathcal{A}_{9,1}^{g_1,g_2,g_3})=Id^G(\mathcal{A}_{9,2}^{g_1,g_2,g_3})=Id^G(\mathcal{A}_{9,3}^{g_1,g_2,g_3})=Id^G(\mathcal{A}_{9}^{g_1,g_2,g_3}),$ we get the desired result.
	\end{proof}
	
	\begin{Lemma} \label{lemmaA_{11}}
	Suppose that $B^{(1,0)}$ contains two orthogonal idempotents $e_1,e_2$, with $e_2\in F\oplus cF$, $c^2=1$, such that
		$j_1e_2j_2e_1j_3\ne 0$ for some  $j_1,j_2,j_3\in J(B)$. Then there exist $g_1,g_2,g_3 \in G$ and a $G\times\mathbb{Z}_2$-graded subalgebra $\bar B$ of $B$  such that $Id^G(E(\bar B))\subseteq Id^G(\mathcal{A}_{10}^{g_1,g_2,g_3})$.
	\end{Lemma}
	\begin{proof}
		By following the proof of Lemma \ref{lemmaA_{10}} we get that,
		if $j_3\in J^{(0)}$ and $j_1, j_2 \in J^{(0)}$ (or $J^{(1)}$),
		then $Id^G(\mathcal{A}_{10}^{g_1,g_2,g_3}) \supseteq Id^G(E(\bar B))$. On the other hand, if $j_3\in J^{(0)}$ and $j_1, j_2$ are of different homogeneous degree with respect to the  induced $\mathbb{Z}_2$-grading,
		then $Id^G(\mathcal{A}_{10,1}^{g_1,g_2,g_3}) \supseteq Id^G(E(\bar B))$. Analogously, if $j_3\in J^{(1)},$ we get that either $Id^G(\mathcal{A}_{10,2}^{g_1,g_2,g_3}) \supseteq Id^G(E(\bar B))$ or $Id^G(\mathcal{A}_{10,3}^{g_1,g_2,g_3}) \supseteq Id^G(E(\bar B))$.
		
		By  Remark \ref{rem1} we have that  $Id^G(\mathcal{A}_{10,1}^{g_1,g_2,g_3})=Id^G(\mathcal{A}_{10,2}^{g_1,g_2,g_3})=Id^G(\mathcal{A}_{10,3}^{g_1,g_2,g_3})=Id^G(\mathcal{A}_{10}^{g_1,g_2,g_3})$, and we are done.
	\end{proof}
	
	\begin{Lemma} \label{lemmaA_{12}}
		If $B^{(1,0)}$ contains two orthogonal idempotents $e_1,e_2$, with $e_2\in F\oplus cF$, $c^2=1$, such that
		$e_1j_1e_2j_2e_1\ne 0$, for some $j_1,j_2\in J(B)$,   then there exist $g_1,g_2 \in G$ and a $G\times\mathbb{Z}_2$-graded subalgebra $\bar B$ of $B$  such that $Id^G(E(\bar B))\subseteq Id^G( \mathcal{A}_{11}^{g_1,g_2})$ or $Id^G(E(\bar B))\subseteq Id^G( \mathcal{A}_{12}^{g_1,g_2})$.
	\end{Lemma}
	
	\begin{proof}
		First assume that the elements  $j_1,j_2$ are  homogeneous of degree $g_1^0 =(g_1,0)$ and $g_2^0 =(g_2,0)$, respectively.
		Let $\bar B$ be the $G\times\mathbb{Z}_2$-graded subalgebra of $B$ generated by the homogeneous elements
		$$
		e_1,  e_2,\ ce_2,\   e_1j_1e_2,\ e_2j_2e_1.
		$$
		We consider the homomorphism of $G \times \mathbb{Z}_2$-graded algebras  $\varphi : \bar B\rightarrow UT(1,2,1)^{(e,g_1^0,g_1^1,(g_1g_2)^0)}$ defined by setting $\varphi(e_1)=e_{11}+e_{44}$, $\varphi(e_2)=e_{22}+e_{33}$, $\varphi(ce_2)=e_{23}+e_{32}$ and $\varphi(e_1j_1e_2)=e_{12}$, $\varphi(e_2j_2e_1)=e_{24}$.
		Then $I=Ker\varphi$ is the ideal of $\bar B$ generated by
		$
		e_2j_2e_1j_1e_2, e_1j_1ce_2j_2e_1
		$
		and $\bar B/I\cong \varphi(\bar B)= P$, where

		$$
		P= \left\{\begin{pmatrix}
			a & b & d & e  \\
			0 & u & v  & h \\
			0 & v & u & m \\
			0 & 0 &  0  & a
		\end{pmatrix} \right\} =  \left\{\begin{pmatrix}
			a & b & 0 & e \\
			0  & u & 0  & h \\
			0 & 0 & u & 0\\
			0 &  0 &  0 & a
		\end{pmatrix}\right\}   \oplus \left\{\begin{pmatrix}
			0 & 0 & d & 0  \\
			0& 0 & v  & 0 \\
			0 & v & 0 & m \\
			0 &  0 & 0  & 0
		\end{pmatrix}\right\}.
		$$
		
		\smallskip
		
		The following elements
		$$
		e_1,  e_2,\ ce_2,\   e_1j_1e_2,\ e_2j_2e_1,\ e_1j_1ce_2,\ ce_2j_2e_1,\ e_1j_1e_2j_2e_1
		$$
		are linearly independent and form a basis of $\bar B$ mod $I.$  We identify these representatives with the corresponding cosets.
		The elements
		$$
		1\otimes e_1,  1\otimes e_2,  1\otimes ce_2, 1 \otimes e_1j_1e_2, 1 \otimes  e_2j_2e_1, 1 \otimes e_1j_1ce_2,  1 \otimes ce_2j_2e_1,  1 \otimes e_1j_1e_2j_2e_1
		$$
		form a basis of $E(\bar B/I)$ isomorphic to $\mathcal{A}_{11}^{g_1,g_2}$.
		Thus  $Id^G( \mathcal{A}_{11}^{g_1,g_2}) \supseteq Id^G(E(\bar B))$.
		
		If $j_1, j_2 \in J^{(1)}$ we obtain the same result.
		If $j_1\in J^{(0)}$ and $j_2 \in J^{(1)}$ (or $j_1\in J^{(1)}$ and $j_2 \in J^{(0)})$, we get a homomorphism of  $G\times\mathbb{Z}_2$-graded algebras $\varphi : \bar B\rightarrow UT(1,2,1)^{(e,g_1^0,g_1^1, (g_1g_2)^1)}$ such that $\bar B/I\cong \varphi(\bar B)= P$ where
		
		$$P=\left\{\begin{pmatrix}
			a & b & d & e  \\
			0& u & v  & h \\
			0 & v & u & m \\
			0 &  0 & 0  & a
		\end{pmatrix}\right\} =  \left\{\begin{pmatrix}
			a & b & 0 & 0  \\
			0& u & 0  & 0 \\
			0& 0 & u & m\\
			0 &0   & 0  & a
		\end{pmatrix}\right\}  \oplus \left\{\begin{pmatrix}
			0 & 0 & d & e  \\
			0 & 0 & v  & h \\
			0 & v & 0 & 0 \\
			0 &  0 &  0 & 0
		\end{pmatrix}\right\}.
		$$
		
		\smallskip
		\noindent
		Its Grassmann envelope is $E(P^{(e,g_1^0,g_1^1, (g_1g_2)^1)})=\mathcal{A}_{12}^{g_1,g_2}$.
		Thus  $Id^G( \mathcal{A}_{12}^{g_1,g_2}) \supseteq Id^G(E(\bar B))$, and we are done.
	\end{proof}

	\section{The main results}
	The aim of  this section is to characterize the varieties of $G$-graded algebras with proper central $G$-exponent greater or equal to two.
	
	\begin{Theorem} \label{teorema1}
		Let $F$ be a field of characteristic zero and $\mathcal{V}$ a variety of  $G$-graded algebras over $F$.
		Then  $exp^{G,\delta}(\mathcal V)> 2$ if and only if  at least one of the G-graded algebras $$\mathcal{A}^{g,i}_1,  \mathcal{A}^p_2, \mathcal{A}_3, \mathcal{A}_4, \mathcal{A}_5^{i,j}, \mathcal{A}_6^{g_1,g_2,g_3}, \mathcal{A}_7^{g_1,g_2,g_3, g_4}, \mathcal{A}_{8}^{g_1,g_2}, \mathcal{A}_{9}^{g_1,g_2,g_3}, \mathcal{A}_{10}^{g_1,g_2,g_3},\mathcal{A}_{11}^{g_1,g_2}, \mathcal{A}_{12}^{g_1,g_2} $$ belongs to $\mathcal{V}$.
	\end{Theorem}
	
	\begin{proof}
	 By Lemma \ref{exponent Ai} and its previous remark it  is clear that, if at least one of the $G$-graded algebras in the statement of the theorem  belongs to $\mathcal{V}$ then $exp^{G,\delta}(\mathcal V)> 2$.
		
	Suppose that  $exp^{G,\delta}(\mathcal V)> 2$. Assume, as we may, that $\mathcal V = var^G(A)= var^G(E(B))$ where $E(B)=E^{(0)}\otimes B^{(0)} \oplus E^{(1)}\otimes B^{(1)}$ is the Grassmann envelope of a finite dimensional
		$G\times {\mathbb Z}_2$-graded algebra $B=B^{(0)}\oplus B^{(1)}$, with  $B^{(0)}=\oplus_{g \in G}B^{(g,0)}$ and $B^{(1)}=\oplus_{g \in G}B^{(g,1)}$.
		Write $B=B_1\oplus \cdots \oplus B_q +J$, with the $B_i$'s simple $G\times {\mathbb Z}_2$-algebras, $1\le i\le q$.
		By Theorem \ref{BSZ},  $B_i= M_{k_i}(F^{\alpha_i} H_i)=F^{\alpha_i} H_i\otimes M_{k_i}(F)$, for $1\le i\le q,$ where $H_i$ is a subgroup of
		$G\times {\mathbb Z}_2$ and  $\alpha_i$ is a  $2$-cocycle.
		
		Suppose first that there exists some $B_i= M_{k_i}(F^{\alpha_i} H_i)$ with $k_i\ge 2$ and grading ${\underline{m_i} }=(e,m_2, \ldots , m_{k_i})\in (G\times {\mathbb Z}_2)^{k_i}.$
		Then $B_i$ contains the graded subalgebra $M_2(F)$   with elementary $G\times {\mathbb Z}_2$-grading induced by ${\underline{m_i}}.$
		Hence
		$$
		E(B)\supseteq E(B_i)\supseteq E(M_2(F)^{\underline{m_i}})
		$$
		and so  $Id^G(E(B))\subseteq Id^G(\mathcal{A}^{g,i}_1),$   for some $g\in G$ and $i\in {\mathbb Z}_2$.
		
		Therefore we may assume that, for $1\le i\le q$, $B_i= F^{\alpha_i} H_i$ with $H_i$ a subgroup of
		$G\times {\mathbb Z}_2$ and  $\alpha_i$ a $2$-cocycle.

		Suppose first that $(g, 0)\in H_i$ for some
		$g\in G$ of order a prime $p>2.$ Let $C_{p,0}=<(g,0)>$, then
		$B_i$ contains the graded subalgebra $F^{\alpha_i}C_{p,0}$  and we may assume that the cocycle $\alpha_i$ is trivial on $C_{p,0}.$
		Thus
		$$
		E(B)\supseteq E(B_i)\supseteq E(FC_{p,0})=E^{(0)}\otimes FC_{p,0}
		$$
		and since $E^{(0)}\otimes FC_{p,0}$ has the same $G$-polynomial identities as $FC_{p}=\mathcal{A}_2^p$ we are done in this case.
		
		Suppose now that $(g, 1)\in H_i$, with $g\in G$ of order a prime $p>2$. Then the element $(g^2, 0)\in H_i$ and we are in the previous case.
		
		Now assume that $H_i$ has order $2^k$, with $k>1$. If there exists $g'\in H_i$ of order $4$, then we
		have two possibilities: either $g' = (g, 0)$ or $g' = (g, 1)$, with  $g$ of order $4$. In the first
		case we get $\mathcal{A}_3 \in \mathcal{V}$ whereas in the second one we have $\mathcal{A}_4 \in \mathcal{V}$.
		
		On the other hand, if there are no elements of order $4$ in $H_i$, then  there exist distinct elements $g'=(g,r)$, $h'=(h,s) \in H_i$,   with $g,h \in G$  distinct elements of order 2.  Let $H_{r,s}=<g',h'>$ be the subgroup of $H_i$ generated by $g'$ and $h'.$
		It is easily checked that $F^{\alpha_i} H_{r,s}$ is a subalgebra of $B_i$ and therefore $\mathcal{A}_5^{r,s}=E(F^{\alpha_i} H_{r,s})\in \mathcal{V}$.
		
		Hence we may assume that, for  $1\le i\le q$, either $B_i= F$, with trivial grading, or $B_i= F\oplus cF$  with  grading $(F, cF)$ and $c^2=1$.

		Let  $D=B_1\oplus \cdots\oplus B_r$ be a centrally admissible $G\times \mathbb{Z}_2$-graded subalgebra of $B$ of maximal dimension, then
		by definition there exists $f$, a proper central $G$-polynomial  of $E(B)$, and an evaluation
		$0\ne f(b_1, \ldots, b_r, b_{r+1}, \ldots, b_n)\in Z(E(B))$, where $b_1\in E(B_1), \ldots, b_r\in E(B_r)$,
		$b_{r+1}, \ldots, b_n\in E(B)$.
		
		Since $exp^{G,\delta}(E(B))= \dim D=d\ge 3$ at least one of the following three cases occurs:
		\begin{itemize}
			\item[1)]
			$D$ contains at least three simple components that must be of the type $F$ or $F\oplus cF$;
			\item [2)]
			$D$ contains exactly one copy of $F$ and one of $F\oplus cF$;
			
			\item[3)]
			$D$ contains only two copies of $F\oplus cF$.
		\end{itemize}
		
		Let's assume first that Case 1 occurs. $D$ contains  three orthogonal idempotents $e_1,e_2,e_3$ homogeneous of degree $(1,0)$.
		Since $f$ is the proper central $G$-polynomial corresponding to $D$ we have that in a non-zero evaluation $\bar f$ of $f$  at least three variables are evaluated in
		$$
		a_1\otimes e_1,\  a_2\otimes e_2,\  a_3\otimes e_3,
		$$
		where $a_1,a_2,a_3\in E^{(0)}$ are distinct elements.
		Since $E^{(0)}$ is central in $E$ and $f$ is multilinear, we may assume that $a_1=a_2=a_3=1$
		and, so,  the elements $1\otimes e_1,\  1\otimes e_2,\ 1\otimes e_3$ appear in the evaluation of $f$.
		
		Suppose first that a monomial of $f$ can be evaluated into a product of the type
		$$
		(1\otimes e_1)(b_1 \otimes j_1)(1\otimes e_2)(b_2 \otimes j_2) (1\otimes e_3)(b_3 \otimes j_3)(1\otimes e_1)\ne 0,
		$$
		for some $b_1\otimes j_1,b_2\otimes j_2,b_3\otimes j_3\in E^{(0)}\otimes  J^{(0)}\cup E^{(1)} \otimes  J^{(1)}.$
		Hence
		$$
		b_1b_2b_3\otimes e_1j_1e_2j_2e_3j_3e_1\ne 0,
		$$
		and this says that $e_1j_1e_2j_2e_3j_3e_1\ne 0$, for some homogeneous elements $j_1,j_2,j_3\in J$ of degree $h_1=(g_1,i_1)$, $h_2=(g_2,i_2)$ and $h_3=(g_3,i_3)$, respectively.
		By Lemma \ref{lemmaA_7} we get that there exists a $G\times \mathbb{Z}_2$-graded subalgebra $\bar B$ of $B$  such that $Id^G(E(\bar B))\subseteq Id^G(\mathcal{A}_6^{g_1,g_2,g_3})$ and we are done in this case. This clearly holds for any permutation of the $e_i$'s, $ i=1,2,3$.
		
		Next suppose that there exists an evaluation of a monomial of $f$ containing a product of the type
		$$
		(b_1 \otimes j_1)(1\otimes e_1)(b_2\otimes j_2) (1\otimes e_2)
		(b_3\otimes j_3)(1\otimes e_3)(b_4 \otimes j_4)\ne 0,
		$$
		for some $b_1\otimes j_1,b_2\otimes j_2,b_3\otimes j_3, b_4\otimes j_4\in E^{(0)}\otimes  J^{(0)}\cup E^{(1)} \otimes  J^{(1)}.$
		Then
		$$
		b_1b_2b_3b_4\otimes j_1e_1j_2e_2j_3e_3j_4\ne 0,
		$$
		and, so, $j_1e_1j_2e_2j_3e_3j_4\ne 0$, for some homogeneous elements $j_1,j_2,j_3, j_4\in J$  of degree $h_t=(g_t,i_t)$,  $1\le t\le 4$.
		By Lemma \ref{lemmaA_8} there exists a $G \times\mathbb{Z}_2$-graded subalgebra $\bar B$ of $B$  such that $Id^G(E(\bar B))\subseteq Id^G(\mathcal{A}_7^{g_1,g_2,g_3, g_4})$ and we are done.
		
		Hence we may assume that in $\bar f$, the evaluation of a monomial of $f$ never starts and ends with elements of $J$.
		Since $\bar f$ is central we may also assume that  $(1\otimes e_i)\bar f=\bar f (1\otimes e_i)=0$, $1\le i\le 3$. Let us set $\bar e_i=1\otimes e_i$, $1\le i\le 3$, and write
		$$
		\bar f = \bar e_1\bar f_1+ \bar f_2 \bar e_1+ \bar e_2 \bar f_3 + \bar f_4 \bar e_2+ \bar e_3\bar f_5+ \bar f_6 \bar e_3,
		$$
		
		\noindent where $\bar e_i\bar f_{2i-1}$  is the sum of the evaluations of the monomials of $f$ starting  with $\bar e_i$, and
		$\bar f_{2i} \bar e_i$ is the  sum of the evaluations of the remaining monomials of $f$ ending with $\bar e_i$, with $i=1,2,3$.

		By the first part of the proof we may assume that  $\bar e_i\bar f_j \bar e_i=0$ for $i \neq j$. Since $\bar e_1\bar f=\bar f \bar e_2=\bar f \bar e_3=0$,  we get
		$$
		0=\bar e_1\bar f \bar e_2=\bar e_1\bar f_1 \bar e_2 + \bar e_1\bar f_4 \bar e_2,
		$$
		$$
		0=\bar e_1\bar f \bar e_3=\bar e_1\bar f_1 \bar e_3 + \bar e_1\bar f_6 \bar e_3.
		$$
		Moreover from
		$$
		\bar e_1\bar f+\bar f \bar e_2+\bar f \bar e_3=0
		$$
		
		\noindent  by using the previous equalities we obtain that

		$$
		\bar e_1\bar f_1 + \bar f_4 \bar e_2 + \bar f_6 \bar e_3=-\bar e_3 \bar f_5 \bar e_2 - \bar e_2 \bar f_3 \bar e_3.
		$$
		
		\bigskip
		\noindent Similarly,
		since $\bar f \bar e_1=\bar e_2 \bar f =\bar e_3 \bar f =0$ we get
		$$
		\bar f_2 \bar e_1 + \bar e_2 \bar f_3 +\bar e_3 \bar f_5 =-\bar e_2 \bar f_6 \bar e_3 - \bar e_3 \bar f_4 \bar e_2.
		$$
		In conclusion we have that
		$$
		\bar f= \bar e_1 \bar f_1+ \bar f_2 \bar e_1  +  \bar e_2 \bar f_3+ \bar f_4 \bar e_2+ \bar e_3 \bar f_5 + \bar f_6 \bar e_3
		=-\bar e_3 \bar f_5 \bar e_2 - \bar e_2 \bar f_3 \bar e_3-\bar e_2 \bar f_6 \bar e_3 - \bar e_3 \bar f_4 \bar e_2.
		$$
		Now, since
		$
		0=\bar f \bar e_2=-\bar e_3\bar f_5 \bar e_2 - \bar e_3 \bar f_4 \bar e_2
		$ and $
		0=\bar f \bar e_3= - \bar e_2 \bar f_3 \bar e_3-\bar e_2 \bar f_6 \bar e_3$
		we obtain
		$\bar f= 0,$
		a contradiction.
		This completes Case 1.

		Now, we consider the second case. Let $D=Fe_1\oplus (F\oplus cF)e_2.$ It is clear that the central polynomial $f$ has a non-zero evaluation $\bar f$ in which
		at least one variable is evaluated in $1\otimes e_1$
		and one variable  in  $(E^{(0)}\otimes e_2)\oplus (E^{(1)}\otimes ce_2).$
		
		Assume that there exists an evaluation of a monomial of $f$ such that
		$$
		(1\otimes e_1)(b_1\otimes j_1)(a_1\otimes e_2)(b_2\otimes j_2)(1\otimes e_1) \ne 0,
		$$
		for some $b_1\otimes j_1,b_2\otimes j_2, \in E^{(0)}\otimes  J^{(0)}\cup E^{(1)} \otimes  J^{(1)}$ and $a_1 \in E^{(0)}$.
		Then
		$$
		b_1a_1b_2\otimes e_1j_1e_2j_2e_1 \ne 0
		$$
		and by Lemma \ref{lemmaA_{12}} we obtain that there exists a $G\times \mathbb{Z}_2$-graded subalgebra $\bar B$ of $B$  such that either
		$Id^G(E(\bar B))\subseteq Id^G(\mathcal{A}_{11}^{g_1,g_2})$ or $Id^G(E(\bar B))\subseteq Id^G(\mathcal{A}_{12}^{g_1,g_2})$ for some $g_1,g_2 \in G$.

		If there exists an evaluation of a monomial of $f$ such that
		$$
		(1\otimes e_1)(b_1\otimes j_1)(a_1\otimes ce_2)(b_2\otimes j_2)(1\otimes e_1) \ne 0,
		$$ for some $b_1\otimes j_1,b_2\otimes j_2, \in E^{(0)}\otimes  J^{(0)}\cup E^{(1)} \otimes  J^{(1)}$ and $a_1 \in E^{(1)}$,
		then
		$
		e_1j_1ce_2j_2e_1 \ne 0.
		$
		If we put $j'_1 = j_1c$
		we are done by the previous case.

		Now suppose there exists an evaluation of a monomial of $f$ such that
		$$
		(a_1\otimes \alpha e_2)(b_1\otimes j_1)(1\otimes e_1)(b_2\otimes j_2)(a_2\otimes \beta e_2)\ne 0,
		$$
		for some $  a_1, a_2 \in E^{(0)}, \alpha, \beta \in \{1, c\},$ $b_1\otimes j_1,b_2\otimes j_2, \in E^{(0)}\otimes  J^{(0)}\cup E^{(1)} \otimes  J^{(1)}.$  Then $e_2j_1e_1j_2e_2 \ne 0$ and,
		by Lemma \ref{lemmaA_9}, there exists a $G\times \mathbb{Z}_2$-graded subalgebra $\bar B$ of $B$  such that $Id^G(\mathcal{A}_{8}^{g_1,g_2}) \supseteq Id^G(E(\bar B))$, for some $g_1,g_2 \in G$.

		Next assume that a monomial of $f$ has an evaluation of the type
		$$
		(b_1 \otimes j_1)(1\otimes e_1)(b_2\otimes j_2) (a_1\otimes e_2)
		(b_3\otimes j_3)\ne 0,
		$$
		for some $b_1\otimes j_1,b_2\otimes j_2,b_3\otimes j_3\in E^{(0)}\otimes  J^{(0)}\cup E^{(1)} \otimes  J^{(1)}, a_1 \in E^{(0)}$.

		Let $j_1e_1j_2e_2j_3\ne 0$. By Lemma \ref{lemmaA_{10}}
		there exists a $G\times \mathbb{Z}_2$-graded subalgebra $\bar B$ of $B$  such that $Id^G(E(\bar B))\subseteq Id^G(\mathcal{A}_{9}^{g_1,g_2,g_3})$.  The same result holds when $j_1e_1j_2ce_2j_3\ne 0$.

		Now we consider the case when a monomial of $f$ has an evaluation of the type
		$$
		(b_1 \otimes j_1)(a_1\otimes e_2)(b_2\otimes j_2) (1\otimes e_1)
		(b_3\otimes j_3)\ne 0,
		$$
		for some $b_1\otimes j_1,b_2\otimes j_2,b_3\otimes j_3\in E^{(0)}\otimes  J^{(0)}\cup E^{(1)} \otimes  J^{(1)}$ and $a_1 \in E^{(0)}.$
		By Lemma \ref{lemmaA_{11}} there exists a $G \times \mathbb{Z}_2$-graded subalgebra $\bar B$ of $B$  such that $Id^G(E(\bar B))\subseteq Id^G(\mathcal{A}_{10}^{g_1,g_2,g_3})$.  The same result holds when
		$j_1ce_2j_2e_1j_3\ne 0$.

		Then, if $\bar m\ne 0$ is the evaluation in $\bar f$ of a monomial of $f$, we may assume that
		$(1\otimes e_i)\bar m (1\otimes e_i)=0$, $1\le i\le 2$, and so
		$(1 \otimes e_i) \bar f (1 \otimes e_i) =0$ and $(1 \otimes e_i) \bar f= \bar f (1 \otimes e_i) =0$, $1\le i\le 2$.
		If we denote $(1 \otimes e_i)=\bar e_i$, then  we may write, as in the previous case,
		$$
		\bar f = \bar e_1\bar f_1+ \bar f_2 \bar e_1+ \bar e_2 \bar f_3 + \bar f_4 \bar e_2.
		$$
		Since $\bar e_1\bar f=0$ we obtain that
		$\bar e_1\bar f_1  + \bar e_1\bar f_4 \bar e_2= 0$
		and so
		$\bar e_1\bar f_1 \bar e_2 + \bar e_1\bar f_4 \bar e_2= 0.$
		It follows that $\bar e_1\Bar{f}_1 = \bar e_1\Bar{f}_1\bar e_2.$
		
		Moreover, since $\bar f\bar e_2 =0$,  similarly we get
		$
		\bar e_1\bar f_1\bar e_2+ \bar f_4 \bar e_2= 0
		$
		and so
		$
		\bar e_1\bar f_1+\bar f_4 \bar e_2= 0.
		$
		As a consequence we obtain that
		$$
		\bar f =  \bar f_2 \bar e_1+ \bar e e_2 \bar f_3.
		$$
		Since
		$0= \bar e_2\bar f =  \bar e_2\bar f_2 \bar e_1+ \bar e_2 \bar f_3,$
		we get $\bar f =  \bar f_2 \bar e_1+ \bar e_2 \bar f_3= \bar f_2 \bar e_1 - \bar e_2\bar f_2 \bar e_1$ and so
		$0=\bar f \bar e_1= \bar f_2 \bar e_1 - \bar e_2\bar f_2 \bar e_1 = \bar f \neq 0$,  a contradiction.
		This completes Case 2.

		We remark that, if $D$ contains two copies of $F\oplus cF$, then it also contains one copy of $F$ and one copy of $F\oplus cF$.
		This says that Case 3 can be deduced from Case 2.
	\end{proof}
	
	In order to  characterize the varieties of $G$-graded algebras with proper central $G$-exponent equal to two  we need to recall the main result proved in \cite{GLP}.
	
	Let $E^b$ denote the Grassmann algebra with the canonical
	$\mathbb{Z}_2$-grading induced by $b\in G$ of order $2$,
	$D^{(1,g,h)}$ and
	$D_0^{(1,g,h,h')}$ the $G$-graded algebras $D=\{(a_{ij})\in UT_3(F) \, | \, a_{11}=a_{33}\}$ and
	$D_0=\{(a_{ij})\in UT_4(F) \, | \, a_{11}=a_{44}=0\}$ with elementary $G$-graded induced by $(1,g,h)$ and  $(1,g,h,h')$,  $g, h, h' \in G,$ respectively.

	\begin{Theorem}\cite[Theorem 3.1]{GLP} \label{Theorem G}
		Let $A$ be a $G$-graded  PI-algebra.
		If   $exp^{G,\delta}(A)\ge 2$, 	then
		$Id^G(A) \subseteq Id^G(\mathcal{A}_2^p)$ or  $Id^G(A) \subseteq Id^G(E)$ or
		$Id^G(A) \subseteq  Id^G(D^{(1,g,h)})$ or
		$Id^G(A) \subseteq  Id^G(D_0^{(1,g',h',l)})$,
		or  $Id^G(A) \subseteq Id^G(E^b)$, in case $|G|$ is even,   for some $a,b, g,h,g',h',l\in G$ and a prime $p$ such that $p\mid |G|$.
	\end{Theorem}
	
	As a consequence we have the following.
	
	\begin{Corollary}
		Let $\mathcal{V}$ be a variety of $G$-graded algebras. Then  $exp^{G,\delta}(\mathcal V)=2$ if and only if  $\mathcal{A}^{g,i}_1,    \mathcal{A}^p_2, \mathcal{A}_3, \mathcal{A}_4, \mathcal{A}_5^{i,j}, \\ \mathcal{A}_6^{g_1,g_2,g_3}, \mathcal{A}_7^{g_1,g_2,g_3, g_4}, \mathcal{A}_{8}^{g_1,g_2}, \mathcal{A}_{9}^{g_1,g_2,g_3}, \mathcal{A}_{10}^{g_1,g_2,g_3}, \mathcal{A}_{11}^{g_1,g_2}, \mathcal{A}_{12}^{g_1,g_2}\not \in \mathcal{V}$  and at least
		one algebra among $D^{(1,g,h)}$,  $D_0^{(1,g,h,h')}$,  $E$ with trivial grading, and $E^b$ if $|G|$ is even, belongs to $\mathcal{V}$.
	\end{Corollary}
	\begin{proof}
		By Theorem \ref{Theorem G} $exp^{G,\delta} (\mathcal V)\geq 2$ if and only if $ \mathcal{A}^p_2=FC_p \in \mathcal V$, with $p\mid |G|$, or $D^{(1,g,h)} \in \mathcal V$ or $D_0^{(1,g,h,h')} \in \mathcal V$ or $E \in \mathcal V$ or $E^b \in \mathcal V$ in case $|G|$ is even,   for some $b, g,h,g',h',l\in G$. Moreover, by Theorem \ref{teorema1}, $exp^{G,\delta} (\mathcal V)\leq 2$ if and only if the $G$-graded algebras  $\mathcal{A}^{g}_1, \ldots, \mathcal{A}_{12}^{g_1,g_2}\not \in \mathcal{V}$, and the proof is completed.
	\end{proof}
	
	\medskip
	Let us recall the following.
	
	\begin{Definition}
		Let $\mathcal{V}$ be a variety of $G$-graded algebras. We say that $\mathcal{V}$ is minimal of proper central $G$-exponent $d$ if exp$^{G,\delta}(\mathcal{V})=d$ and for every proper subvariety $\mathcal{U}\subset \mathcal{V}$ we have that exp$^{G,\delta}(\mathcal{U})<d$.
	\end{Definition}
	
	\medskip
	
	Let   $var^G(A)$ be the variety of $G$-graded algebras generated by an algebra $A$.
	Then denote
	$$
	\mathcal{V}_1^g=var^G(\mathcal{A}^{g,i}_1), \
, \
	\mathcal{V}_2^p=var^G(\mathcal{A}^{p}_2), \
	\mathcal{V}_3=var^G(\mathcal{A}_3), \
	\mathcal{V}_4=var^G(\mathcal{A}_4), \
	\mathcal{V}_5^{i,j}=var^G(\mathcal{A}^{i,j}_5),
	$$
	$$
	\mathcal{V}_6^{g_1,g_2,g_3}=var^G( \mathcal{A}_6^{g_1,g_2,g_3}), \ \mathcal{V}_7^{g_1,g_2,g_3,g_4}=var^G( \mathcal{A}_7^{g_1,g_2,g_3,g_4}),\ \mathcal{V}_8^{g_1,g_2}=var^G( \mathcal{A}_8^{g_1,g_2}),
	$$
	$$
	\mathcal{V}_{9}^{g_1,g_2,g_3}=var^G( \mathcal{A}_{9}^{g_1,g_2,g_3}), \
	\mathcal{V}_{10}^{g_1,g_2,g_3}=var^G( \mathcal{A}_{10}^{g_1,g_2,g_3}), \ \mathcal{V}_{11}^{g_1,g_2}=var^G( \mathcal{A}_{11}^{g_1,g_2}), \
	\mathcal{V}_{12}^{g_1,g_2}=var^G( \mathcal{A}_{12}^{g_1,g_2}).
	$$
	
	\bigskip
	
	We say that two varieties $\mathcal V$ and $\mathcal W$ are not comparable if $\mathcal V\nsubseteq \mathcal W$ and
	$\mathcal W\nsubseteq  \mathcal V$. Then we have the following.
	
	\begin{Proposition} \label{differentvarieties}
		The varieties   $\mathcal{V}_1^{g,i},  \mathcal{V}_2^p, \mathcal{V}_3, \mathcal{V}_4, \mathcal{V}_5^{i,j},$ $\mathcal{V}_6^{g_1,g_2,g_3}, \mathcal{V}_7^{g_1,g_2,g_3,g_4}, \mathcal{V}_8^{g_1,g_2},$ $\mathcal{V}_{9}^{g_1,g_2,g_3}$ and $\mathcal{V}_{10}^{g_1,g_2,g_3}$ are not comparable. Moreover they do not contain the varieties $\mathcal{V}_{11}^{g_1,g_2}, \mathcal{V}_{12}^{g_1,g_2}$.
	\end{Proposition}
	\begin{proof}

   We start by comparing among themselves varieties of the same type. Clearly the varieties of the type $\mathcal{V}_1^{g,i}$  are not comparable among themselves.
 Moreover, since the variety $\mathcal{V}_2^p$ has almost polynomial $\delta$-growth (see \cite{GLP}), we immediately get to   $\mathcal{V}_2^p$ and $\mathcal{V}_2^q$ are not comparable, for $p \ne q$.
		
Let us now consider the varieties of type $\mathcal{V}_5^{i,j}$. We start by comparing the graded identities of $\mathcal{V}_5^{i,j}$ and $\mathcal{V}_5^{i',j'}$ when $(i',j')\ne (i,j)$ and $(i',j')\ne (0,0)$.
	 Let  $\tilde{g}\in G$ be an element of order two such that $(\mathcal{A}^{i,j}_5)^{(\tilde{g})}=E^{(0)}\otimes F^\alpha(\tilde{g},0)$ and $(\mathcal{A}^{i',j'}_5)^{(\tilde{g})}=E^{(1)}\otimes F^\alpha(\tilde{g},1)$. Then
		$x_1^{\tilde{g}}\circ x_2^{\tilde{g}}=x_1^{\tilde{g}} x_2^{\tilde{g}}+x_2^{\tilde{g}}x_1^{\tilde{g}}$ is an identity of $\mathcal{A}^{i',j'}_5$ but not of $\mathcal{A}^{i,j}_6$.
        Now, let $(i',j')=(0,0)$ and let $\tilde{g}, \tilde{h} \in G$ be elements of order two such that $(\mathcal{A}^{i,j}_5)^{(\tilde{g})}=E^{(1)}\otimes F^\alpha(\tilde{g},1)$ and $(\mathcal{A}^{i,j}_5)^{(\tilde{h} )}=E^{(1)}\otimes F^\alpha(\tilde{h} ,1)$. Then $\alpha(\tilde{g},\tilde{h} )x_1^{\tilde{h} }x_2^{\tilde{g}}-\alpha(\tilde{h} ,\tilde{g})x_1^{\tilde{g}}x_2^{\tilde{h} }$  is an identity of $\mathcal{A}^{0,0}_5$ but not of $\mathcal{A}^{i,j}_5$.  Thus the varieties of type $\mathcal{V}_5^{i,j}$ are not comparable among themselves.

		Let's go on by comparing the graded identities of $\mathcal{A}_6^{g_1,g_2,g_3}$ and $\mathcal{A}_6^{h_1,h_2,h_3}$ when $(g_1,g_2,g_3)\ne (h_1,h_2,h_3)$.
		Since $e_{12}e_{23}e_{34}\ne 0$ we get that $x_1^{g_1}x_2^{g_2}x_3^{g_3}$ is an identity of  $\mathcal{A}_6^{h_1,h_2,h_3}$ but  not of  $\mathcal{A}_6^{g_1,g_2,g_3}$. Similarly $x_1^{h_1}x_2^{h_2}x_2^{h_3}$ is an identity of  $\mathcal{A}_6^{g_1,g_2,g_3}$ but  not of  $\mathcal{A}_6^{h_1,h_2,h_3}$.
		Then $Id^G(\mathcal{A}_6^{g_1,g_2,g_3})\not\subseteq Id^G(\mathcal{A}_6^{h_1,h_2,h_3})$
		and $Id^G(\mathcal{A}_6^{h_1,h_2,h_3})\not\subseteq Id^G(\mathcal{A}_6^{g_1,g_2,g_3})$.
		Thus $\mathcal{V}_6^{g_1,g_2,g_3}$ and $\mathcal{V}_6^{h_1,h_2,h_3}$ are not comparable.
		
		We can apply a similar consideration  to the $G$-graded algebras of the type $\mathcal{A}_7^{g_1,g_2,g_3,g_4}$ by considering the product $e_{12}e_{23}e_{34}e_{45}\ne 0$ and we get that  $\mathcal{V}_7^{g_1,g_2,g_3,g_4}$ and $\mathcal{V}_7^{h_1,h_2,h_3,h_4}$ are not comparable.
		
		Now, we consider the $G$-graded algebras of type $\mathcal{A}_8^{g_1,g_2}$. In this case $\alpha \beta e_{12}e_{23}\ne 0$, for any $\alpha, \beta \in E$, then $x_1^{g_1}x_2^{g_2}$
		is not an identity of $\mathcal{A}_8^{g_1,g_2}$ but  is an identity of $\mathcal{A}_8^{h_1,h_2}$ when $(g_1,g_2)\ne (h_1,h_2)$ and $Id^G(\mathcal{A}_8^{h_1,h_2})\not\subseteq Id^G(\mathcal{A}_8^{g_1,g_2})$. Similarly
		we prove that $Id^G(\mathcal{A}_8^{g_1,g_2})\not\subseteq Id^G(\mathcal{A}_8^{h_1,h_2})$ and so $\mathcal{V}_8^{g_1,g_2}$ and $\mathcal{V}_8^{h_1,h_2}$ are not comparable.
		
		Now, we compare the graded identities of $\mathcal{A}_{9}^{g_1,g_2,g_3}$ and $\mathcal{A}_{9}^{h_1,h_2,h_3}$ when $(g_1,g_2,g_3)\ne (h_1,h_2,h_3)$. We notice that
		$e_{12}e_{23}e_{35}\ne 0$ then $x_1^{g_1}x_2^{g_2}x_3^{g_3} \notin Id^G(\mathcal{A}_{9}^{g_1,g_2,g_3})$ but $x_1^{g_1}x_2^{g_2}x_3^{g_3} \in Id^G(\mathcal{A}_{9}^{h_1,h_2,h_3})$. Also $x_1^{h_1}x_2^{h_2}x_3^{h_3} \notin Id^G(\mathcal{A}_{9}^{h_1,h_2,h_3})$ but $x_1^{h_1}x_2^{h_2}x_3^{h_3} \in Id^G(\mathcal{A}_{9}^{g_1,g_2,g_3})$. Thus $\mathcal{V}_{9}^{g_1,g_2,g_3}$ and $\mathcal{V}_{9}^{h_1,h_2,h_3}$ are not comparable.
		
		By using similar arguments we obtain that $\mathcal{V}_{10}^{g_1,g_2,g_3}$ and $\mathcal{V}_{10}^{h_1,h_2,h_3}$ are not comparable.

		Now, we have to compare varieties of different types. We observe that
		if we regard the ordinary identities as $G$-polynomial identities, then by \cite[Proposition 1]{BV} we get that
		$\mathcal{V}_1^{g,i},$ $\mathcal{V}_6^{g_1,g_2,g_3}, \mathcal{V}_7^{g_1,g_2,g_3,g_4}, \mathcal{V}_8^{g_1,g_2},$ $\mathcal{V}_9^{g_1,g_2,g_3}$ and $\mathcal{V}_{10}^{g_1,g_2,g_3}$ are not comparable varieties.
		
		Hence we only need to compare the varieties $\mathcal{V}_2^p, \mathcal{V}_3, \mathcal{V}_4$ and $ \mathcal{V}_5^{i,j}$   with
		other types.
		
		Let $\mathcal V$ be one of the above varieties. Since $exp^{G,\delta}(\mathcal{V})=3,4$ and $\mathcal{V}_2^p$ has almost polynomial $\delta$-growth (see \cite{GLP}), for $p \ne 3$, the varieties $\mathcal{V}_2^p$ and $\mathcal V$ are not comparable.
		
		Now, if $p=3,$ we shall consider the $C_3$-graded algebra $\mathcal{A}_2^3$ where $C_3= \langle h \rangle,$ $h \in G.$ We observe that $\mathcal{A}_2^3 \notin \mathcal{V}_1^{g,i}$; in fact, if $\mathcal{A}_1^{g,i}$ has trivial grading, then $\mathcal{A}_2^3$  has more
		non-zero homogeneous components and we are done. Otherwise we have to
		consider only the cases in which $g = h$ or $g = h^2$. In both cases it is clear
		that $x_1^gx_2^g \equiv 0$ in $\mathcal{A}_1^{g,i}$  but not in $\mathcal{A}_2^3$.
		Since $exp^{G,\delta}(\mathcal{V}_1^{g,i})=4$ and $exp^{G,\delta}(\mathcal{V}_2^3)=3$ we get that the varieties $\mathcal{V}_2^3$ and $\mathcal{V}_1^{g,i}$ are not comparable.
		
		It should be notice  that if two algebras are graded by different nonisomorphic groups then there is nothing to prove. This is the case of the algebras $\mathcal{A}_2^3$
		and $B$, where $B \in \{ \mathcal{A}_3, \mathcal{A}_4, \mathcal{A}_5^{i,j}\}$.
		Hence we only need to compare the variety $\mathcal{V}_2^3$ with the
		variety  $\mathcal{V}=var^G(\mathcal{A})$,  where $\mathcal{V} \in \{ \mathcal{V}_6^{g_1,g_2,g_3}, \mathcal{V}_7^{g_1,g_2,g_3, g_4},  \mathcal{V}_8^{g_1,g_2}, \mathcal{V}_{9}^{g_1,g_2,g_3},\mathcal{V}_{10}^{g_1,g_2,g_3} \}$. This is achieved by noticing that  $\mathcal{A}_2^3$ is commutative and the
		components of $\mathcal{A}$  of  homogeneous degree different from one are nilpotent of index at most three.
		
		Since $\mathcal A_3, \mathcal A_4, \mathcal{A}_5^{i,j}$ have  more
		non-zero homogeneous components than  $\mathcal{A}_1^{g,i}$ it follows that $\mathcal{A}_3, \mathcal{A}_4, \mathcal{A}_5^{i,j}\not \in \mathcal{V}_1^{g,i}.$

        Now, we shall consider the variety $\mathcal{V}_3$. We start by comparing the $G$-polynomial identities of $\mathcal{A}_3$ and of $\mathcal{A}_1^{g,0}$.  If $g=1_G$ then $[x_1^1,x_2^1]$ is an identity of $\mathcal{A}_3$ but is not an identity of $\mathcal{A}_1^{g,0}$.
		If $g=h^2$ with $o(h) = 4$, then $[x_1^g,x_2^g]\in Id^G(\mathcal{A}_3)$ and $[x_1^g,x_2^g]\notin Id^G(\mathcal{A}_1^{g,0})$. If $g=h$ or $g=h^3$ with $o(h) = 4$, then $[x_1^1,x_2^g]\in Id^G(\mathcal{A}_3)$ and $[x_1^1,x_2^g]\notin Id^G(\mathcal{A}_1^{g,0})$.
		Thus the varieties $\mathcal{V}_3$ and $\mathcal{V}_1^{g,0}$ are not comparable.
		Similarly it can be proved that $\mathcal{V}_3$ and $\mathcal{V}_1^{g,1}$ are not comparable.

        The calculations for the variety $\mathcal{V}_4$  are the same as those made for the variety $\mathcal{V}_3$. The only difference is to consider the $G$-graded polynomial $x_1^1 \circ x_2^g$ instead of the polynomial  $[x_1^1,x_2^g]$ in the comparison with the variety $\mathcal{V}_1^{g,i}$ with $g=h$ or $g=h^3$ and $o(h) = 4$.

        Let's now consider the variety  $\mathcal{V}_5^{i,j}$; we have that
		  $\mathcal{A}_1^{g,i} \notin \mathcal{V}_5^{i,j}.$ In fact if $g={h_1}$ or $g={h_2}$ or $g={h_1}{h_2}$
		then $[x_1^1,x_2^g]$ is an identity of $\mathcal{A}_5^{i,j}$ but not of $\mathcal{A}_1^{g,i},$  where $h_1, h_2 \in G$ are distinct elements order 2. Otherwise, we have to consider $[x_1^1,x_2^1]$. Thus $\mathcal{V}_5^{i,j}$ is not comparable with the varieties $\mathcal{V}_1^{g,i}.$

		It is clearly that $\mathcal{V}_3$ and $\mathcal{V}_4$ are not comparable, in fact $[x_1^g,x_2^g]\in Id(\mathcal{A}_3)$ but   $[x_1^g,x_2^g]\notin Id(\mathcal{A}_4)$ and  $x_1^g\circ x_2^g=x_1^gx_2^g+x_2^gx_1^g \in Id(\mathcal{A}_4)$
		but $x_1^g\circ x_2^g \notin Id(\mathcal{A}_3)$.
		Moreover $\mathcal{V}_3$ and $\mathcal{V}_5^{i,j}$ are not comparable because they are graded by different non isomorphic groups.
		If we take a variety of $G$-graded algebras  $\mathcal{V}=var^G(\mathcal A)$, where $\mathcal{V} \in \{ \mathcal{V}_6^{g_1,g_2,g_3}, \mathcal{V}_7^{g_1,g_2,g_3, g_4}, \mathcal{V}_8^{g_1,g_2}, \mathcal{V}_{9}^{g_1,g_2,g_3},\mathcal{V}_{10}^{g_1,g_2,g_3} \}$, then $\mathcal{V}_3$ and $\mathcal{V}$ are not comparable since $\mathcal{A}_3$ is commutative and the
		components of $\mathcal A$  of  homogeneous degree different from one are nilpotent of index at most three.

		  Similarly considerations hold for the variety $\mathcal{V}_4.$
		
		Finally, let's consider the variety  $\mathcal{V}_5^{i,j}$  and let
		 $\mathcal{V}=var^G(\mathcal{A})$,  where $\mathcal{V} \in \{ \mathcal{V}_6^{g_1,g_2,g_3}, \mathcal{V}_7^{g_1,g_2,g_3, g_4}, \mathcal{V}_8^{g_1,g_2}, \mathcal{V}_{9}^{g_1,g_2,g_3},\mathcal{V}_{10}^{g_1,g_2,g_3} \}$. Then, as for the variety $\mathcal{V}_3$, by remarking that $\mathcal{A}_5^{i,j}$ is commutative and the
		components of $\mathcal{A}$  of  homogeneous degree different from one are nilpotent of index at most three,  we obtain that $\mathcal{V}_5^{i,j}$ and $\mathcal{V}$ are not comparable.
		
		The first part of the proposition is  proved.
		
		Now  consider the variety $\mathcal{V}_{11}^{g_1,g_2}$.
		From \cite[Proposition 1]{BV}, by regarding the ordinary identities as $G$-polynomial identities,  we get that
		$\mathcal{V}_{11}^{g_1,g_2} \not\subseteq \mathcal{V}$ where $\mathcal{V} \in \{ \mathcal{V}_1^{g,i}, \mathcal{V}_6^{g_1,g_2,g_3}, \mathcal{V}_7^{g_1,g_2,g_3, g_4}, \mathcal{V}_8^{g_1,g_2}, \mathcal{V}_9^{g_1,g_2,g_3},\mathcal{V}_{10}^{g_1,g_2,g_3} \}$.
		
		Moreover $\mathcal{V}_{11}^{g_1,g_2} \not\subseteq \mathcal{V}_2^p$; in fact it is sufficient to remark that the $G$-graded algebra $\mathcal{A}_2^p$ is abelian and the polynomial $[x_1^{1},x_2^{g_1}]$ is not an identity of $\mathcal{A}_{11}^{g_1,g_2}$.
		By the same argument we get that $\mathcal{V}_{11}^{g_1,g_2} \not\subseteq \mathcal{V}_3$.
		Also we have that
		$\mathcal{V}_{11}^{g_1,g_2} \not\subseteq \mathcal{V}_4$, in fact it is easy to see that  $x_1^{1} \circ x_2^{g_1}\in Id^G(\mathcal{A}_4)$  and  $x_1^{1} \circ x_2^{g_1} \notin Id^G(\mathcal{A}_{11}^{g_1,g_2})$.
		Finally, $\mathcal{V}_{11}^{g_1,g_2} \not\subseteq \mathcal{A}_5^{i,j}$ since
		$[x_1^{1},x_2^{g_1}]\in Id^G(\mathcal{A}_5^{i,j})$ and  $[x_1^{1},x_2^{g_1}]\notin Id^G(\mathcal{A}_{11}^{g_1,g_2})$.
		
		The proof is now complete.
	\end{proof}
	
	\medskip
	
	As a consequence we get the following.
	
	\medskip
	
	\begin{Corollary} The variety  $\mathcal{V}_2^p$, with $p\mid |G|$, is minimal  of proper central $G$-exponent equal to $p.$ The varieties  $\mathcal{V}_1^{g,i}$, $\mathcal{V}_3$, $\mathcal{V}_4$ and $\mathcal{V}_5^{i,j}$ are minimal  of proper central $G$-exponent equal to $4.$
		The varieties $\mathcal{V}_6^{g_1,g_2,g_3}, \mathcal{V}_7^{g_1,g_2,g_3, g_4},  \mathcal{V}_8^{g_1,g_2}, \mathcal{V}_{9}^{g_1,g_2,g_3}$ and $\mathcal{V}_{10}^{g_1,g_2,g_3}$  are minimal of proper central $G$-exponent equal to $3.$
	\end{Corollary}
	\begin{proof} By Lemma \ref{exponent Ai}, $exp^{G,\delta}(\mathcal{V}_2^p)=p$,
		$exp^{G,\delta}(\mathcal{V}_1^{g,i})=$  $exp^{G,\delta}(\mathcal{V}_3)=$ $exp^{G,\delta}(\mathcal{V}_4)=$ $exp^{G,\delta}(\mathcal{V}_5^{i,j})=4$ and   $exp^{G,\delta}(\mathcal{V}_6^{g_1,g_2,g_3})=$ $exp^{G,\delta}(\mathcal{V}_7^{g_1,g_2,g_3, g_4})=$ $exp^{G,\delta}(\mathcal{V}_8^{g_1,g_2})=$ $exp^{G,\delta}(\mathcal{V}_{9}^{g_1,g_2,g_3})=$ $exp^{G,\delta}(\mathcal{V}_{10}^{g_1,g_2,g_3})=3$.
		
		Now,
		let $\mathcal{V} \in  \{ \mathcal{V}_1^{g,i},  \mathcal{V}_2^p, \mathcal{V}_3, \mathcal{V}_4, \mathcal{V}_5^{i,j}, \mathcal{V}_6^{g_1,g_2,g_3}, \mathcal{V}_7^{g_1,g_2,g_3, g_4}, \mathcal{V}_8^{g_1,g_2}, \mathcal{V}_{9}^{g_1,g_2,g_3},\mathcal{V}_{10}^{g_1,g_2,g_3} \}$ and let $\mathcal{U}$ be a proper subvariety of $\mathcal{V}$.
		Suppose  that $exp^{G,\delta}(\mathcal{U})> 2$.
		By Theorem \ref{teorema1} and  Proposition \ref{differentvarieties},  we obtain  a contradiction.
	\end{proof}


\bigskip
	
	
	
	\begin{thebibliography}{99}
		
		
		\bibitem{AG} E. Aljadeff, A. Giambruno, {\em Multialternating graded polynomials and growth of polynomial identities}, Proc. Am. Math. Soc. {\bf 141} (2013), 3055–3065.
		
		\bibitem{AKB} E. Aljadeff, A. Kenal–Belov, {\em Representability and Specht problem for G-graded algebras}, Adv.
		Math. {\bf 225}  (2010), 2391–2428.
		
		
		\bibitem{AGL} E. Aljadeff, A. Giambruno, D. La Mattina, {\em Graded polynomial identities and exponential growth}, J. Reine Angew. Math. {\bf 650} (2011), 83-100.
		
		\bibitem{BSZ} Y. A. Bahturin, S. K. Sehgal, M. V. Zaicev, {\em Finite-dimensional simple graded algebras}, Sb. Math. {\bf 199} (2008), no. 7, 965–983.
		
		\bibitem{BV} F. S. Benanti, A. Valenti, {\em A Characterization of varieties of algebras of proper central
			exponent equal to two}, arXiv:2412.17689 [math.RA].
		
		\bibitem{CM} M. Cohen, S. Montgomery, {\em  Group-graded rings, smash product and group actions},
		Trans. Amer. Math. Soc. {\bf 282} (1984), 237–258.
		
		\bibitem{CR} C. W. Curtis, I. Reiner, Representation Theory of Finite Groups and Associative Algebras, Wiley Classics Library, John Wiley \& Sons, New York, 1988.
		
		
		\bibitem{DVLS} O.M. Di Vincenzo, R. La Scala,  {\em  Block-triangular matrix algebras and factorable ideals of graded polynomial identities},  J. Algebra {\bf 279} (2004), 260--279.
		
		\bibitem{Formanek} E. Formanek, {\em Central polynomials for matrix rings}, J. Algebra {\bf 23} (1972), 129--132.
		
		
		\bibitem{GL} A. Giambruno, D. La Mattina, {\em Graded polynomial identities and codimensions: computing the exponential growth}, Adv. Math. {\bf 225} (2010), 859–881.
		
		
		\bibitem{GLP} A. Giambruno, D. La Mattina, C. Polcino Milies, {\em  On almost polynomial growth of proper central polynomials}, Proc. Amer. Math. Soc. {\bf 152} (2024), 4569--4584.
		
		
		
		\bibitem{GR} A. Giambruno, A. Regev,
		{\em Wreath products and P.I. algebras},   J. Pure Appl. Algebra {\bf 35} (1985), 133--149.
		
		\bibitem{GZ1} A. Giambruno,  M. Zaicev, {\em  On Codimension growth of finitely generated associative algebras},
		Adv. Math. {\bf 140} (1998), 145--155.
		
		\bibitem{GZ2} A. Giambruno,  M. Zaicev, {\em Exponential codimension growth of P.I. algebras: An exact estimate}, Adv. Math. {\bf 142} (1999), 221--243.
		
		\bibitem{GZ2003AM} A. Giambruno, M. Zaicev,
		{\em Minimal varieties of algebras of exponential growth},
		Adv. Math. {\bf 174} (2003), 310–323.
		
		\bibitem{GZbook} A. Giambruno, M. Zaicev,
Polynomial Identities and Asymptotic Methods, AMS, Mathematical
Surveys and Monographs Vol. {\bf 122}, Providence, R.I., 2005.
\bibitem{GZ2018} A. Giambruno, M. Zaicev,
		{\em Central polynomials and growth functions},  Israel J. Math. {\bf 226} (2018), 15--28.
		
		\bibitem{GZ2019} A. Giambruno, M. Zaicev,
		{\em Central polynomials of associative algebras and their growth}, Proc. Amer. Math. Soc.  {\bf 147}, Number 3 (2019),  909–919.
		
		
		\bibitem{IM} A. Ioppolo,  F. Martino, {\em Classifying G-graded algebras of exponent two}, Isr. J. Math. {\bf 229} (2019), 341–356.
		
		
		\bibitem{Ka} Y. Karasik, {\em Kemer’s theory for H-module algebras with application to the PI exponent},
		J. Algebra {\bf 457} (2016), 194--227.
		
		\bibitem{kemer} A. R. Kemer, Ideals of Identities of
		Associative Algebras, AMS Translations of Mathematical Monograph,
		Vol. {\bf 87}, 1988.
		
		\bibitem{LMR}  D. La Mattina, F. Martino, C. Rizzo, {\em Central polynomials of graded algebras: capturing their exponential growth}, J. Algebra {\bf 600}  (2022), 45--70.
		
		\bibitem{Razmyslov} Ju. P. Razmyslov, {\em A certain problem of Kaplansky}, (Russian)
		Izv. Akad. Nauk SSSR Ser. Mat. {\bf 37} (1973), 483--501.
		
		\bibitem{SV} D. Stefan, F. Van Oystaeyen, {\em The Wedderburn–Malcev theorem for comodule algebras}, Commun. Algebra {\bf 27} (1999), 3569--3581.
		
		\bibitem{Sv} I. Sviridova, {\em Identities of PI-algebras graded by a finite Abelian group}, Commun. Algebra {\bf 39} (2011), 3462--3490.
		
		\bibitem{V} A. Valenti, {\em Group graded algebras and almost polynomial growth}, J. Algebra {\bf 334} (2011), 247--254.
		
	\end{thebibliography}
\end{document}